\documentclass[final,10pt]{siamltex}
\usepackage[latin1]{inputenc}
\usepackage{amsmath,amstext,amsbsy,amsopn,amscd,amsxtra,upref,amssymb}
\usepackage{amsfonts}
\usepackage{graphicx}
\usepackage{algorithmic}
\usepackage{algorithm}
\usepackage[mathscr]{euscript}
\usepackage{pgf,tikz}
\usetikzlibrary{arrows}
\newtheorem{remark}[theorem]{Remark}

\newcommand{\Id}{\mathbb I}
\newcommand{\SO}{\mathbb S}
\newcommand{\WSO}{\mathscr S}

\newcommand{\WU}{\widehat{U}}

\newcommand{\D}{\mathbb D}
\newcommand{\f}{\mathbf f}

\newcommand{\0}{\mathbf 0}
\newcommand{\R}{\mathbb R}
\newcommand{\Prj}{\mathcal P}
\newcommand{\WP}{\widehat{{\mathcal P}}}

\newcommand{\mat}[1]{{#1}}

\definecolor{bleu1}{RGB}{0,51,128}

\title{The Discrete Empirical Interpolation Method: Canonical Structure and Formulation in
Weighted Inner Product Spaces\thanks{The work of the first author was supported by grant
		HRZZ-9345 from the Croatian Science Foundation. 	}}

\author{Zlatko Drma\v{c}\thanks{Faculty of Science, Department of Mathematics, University of Zagreb,
		Bijeni\v{c}ka 30, 10000 Zagreb, Croatia, drmac@math.hr} 
\and Arvind Krishna Saibaba\thanks{Department of Mathematics, North Carolina State University,
Raleigh, NC 27695-8205, USA, asaibab@ncsu.edu}
}

\begin{document}
	\maketitle
\begin{abstract}
New  contributions are offered to the theory and numerical implementation of the 
Discrete Empirical Interpolation Method (DEIM). {
A substantial tightening of the error bound for the DEIM oblique 
projection is achieved by index selection via a strong rank revealing 
QR factorization. This removes the exponential factor in the dimension of the 
search space from the DEIM projection error, and allows sharper a priori error bounds. Well-known canonical structure of pairs of projections is used to reveal canonical structure of DEIM.
Further, the DEIM approximation is formulated in weighted inner product 
defined by a real symmetric positive-definite matrix $W$. 
The weighted DEIM ($W$-DEIM) can be interpreted  as a numerical implementation of the  
Generalized Empirical Interpolation Method (GEIM)
and the more general Parametrized-Background Data-Weak (PBDW) approach. 
Also, it can be naturally deployed in the framework when the POD Galerkin projection is formulated in a 
discretization of a suitable energy (weighted) inner product such that the  
projection preserves important physical properties such 
as e.g. stability. While the theoretical foundations of weighted POD and the GEIM 
are available in the more general setting of function spaces, this paper focuses 
to the gap between sound functional analysis and the core numerical linear algebra. 
The new proposed algorithms allow different forms of $W$-DEIM for point-wise and generalized interpolation. For the generalized interpolation, our bounds show that the condition number of $W$ does not affect the accuracy, and for point-wise interpolation the condition number of the weight matrix $W$ enters the bound essentially  as 
$\sqrt{\min_{D=\mathrm{diag}}\kappa_2(DWD)}$, where $\kappa_2(W)=\|W\|_2 \|W^{-1}\|_2$ is the spectral condition number.  }
\end{abstract}



\section{Introduction}
%
%

Suppose we want to run numerical simulations of a physical reality described by a set of ordinary differential equations (ODEs)
\begin{equation}\label{eq:LTI}
\dot x(t) = A x(t) + \f(x(t)),\;\;x(0)=x_0\in\R^m,
\end{equation}
where $A\in\R^{m\times m}$, $\f:\R^m\longrightarrow \R^m$. Often, of interest is $y(t)=C x(t)$, with some given $p\times m$ matrix $C$.
Such a system of ODEs can arise from discretization of a spatial differential operator in time dependent PDEs (e.g. method of lines), e.g. for the purposes of prediction and/or control, or optimization with respect to a set of parameters. In a parameter dependent case we have $A=A(\mu)$, $x=x(t;\mu)$, $x_0=x_0(\mu)$, and $\f(\cdot;\mu)$ is also parameter dependent, where the parameter $\mu$, that may carry e.g. information on material properties, is from a parameter domain $\mathcal{P}\subset \R^d$, $d\geq 1$.\footnote{To keep the notation simple, we suppress the explicit parameter dependence until numerical experiments in \S \ref{S=Examples}.} The function $\f(\cdot;\cdot)$ is in general assumed to be nonlinear.   
Large dimension $m$ (say, $m>10^5$) makes the task computationally intractable {for multiple query problems},  and one is forced to devise and use a reduced order system that emulates (\ref{eq:LTI}).

In a projection based model order reduction, one constructs a suitable low dimensional subspace $\mathcal{V}_k$ as the range of an $m\times k$ orthonormal matrix $V_k$ ($V_k^T V_k=\Id_k$) and seeks an approximation of the form $x(t) \approx \overline{x} + V_k \widehat{x}(t)$, $\widehat{x}\in\R^k$. The solution $x(t)$ is stored at a set of discrete times (also known as snapshots) and $\overline{x}$ is the average over the snapshots. The matrix $V_k$ can be, e.g., the POD basis of the $k$ leading left singular vectors of the centered snapshots $x(t_i)-\overline{x}$, computed at the discrete times $t_i$ from high resolution numerical simulations in the off-line phase; possibly over a parameter grid.  It is assumed that $k\ll m$.
By enforcing the orthogonality of the residual and the space $\mathcal{V}_k$, one obtains Galerkin projection of the original problem
\begin{equation}\label{eq:G1}
\dot{\widehat{x}}(t) = {V_k^T A V_k} \widehat{x}(t) + V_k^T A \overline{x}+ V_k^T \f(\overline{x}+V_k \widehat{x}(t)),\;\;\widehat{x}(0)= V_k^T (x(0)-\overline{x}), 
\end{equation}
where $A_k=V_k^T A V_k$  is $k\times k$, $V_k^T A \overline{x}\in\R^k$, but the projected nonlinear forcing term $V_k^T \f(\overline{x}+V_k \widehat{x}(t))$ still involves the dimension $m$,  in computing $\widetilde{x}(t)=\overline{x}+V_k \widehat{x}(t)$ and  $f=\f(\widetilde{x}(t))$ (at a sequence of discrete values $t=t_i$), as well as in computing $V_k^T f$. 
For large $m$, this carries substantial computational effort and heavy memory traffic. 

The Discrete Empirical Interpolation (DEIM) \cite{DEIM} method provides a way to alleviate these burdens, and to efficiently approximate $\f(\cdot)$ from a learned subspace. DEIM originates in the Empirical Interpolation Method (EIM) \cite{grepl-maday-2007}, \cite{EIM}, \cite{Maday2009383} and it uses the reduced basis provided by the POD. For a related discrete version of EIM see \cite{Haasdonk-Ohlberger-Rozza-2008}. {Here, for the reader's convenience, we first briefly review the main steps of DEIM approximation and its error estimate, and then we place it in the more general concepts of GEIM and PBDW.}

\subsection{DEIM}\label{s_deim}
Suppose we have empirically determined an $r$-dimensional subspace $\mathcal{U}_r$ as the range of  an orthonormal $U_r$ such that $U_r U_r^T \f(\overline{x}+V_k \widehat{x}(t)) \approx \f(\overline{x}+V_k \widehat{x}(t))$. This can be done e.g. by the POD, which will determine a suitable dimension $r$ from the decay of the singular values of the matrix of snapshots. 
The tacit assumption is that $\f(\cdot)$ is from a set of functions $\mathcal{F}$ with small Kolmogorov $r$-width \cite{Kolmogorov-1936}, \cite[Chapter 6]{kowalski-sikorski-stenger-1995}, \cite{Maday2009383}.
Inserting the orthogonal projection $U_r U_r^T$ into (\ref{eq:G1}) gives
\begin{equation}\label{eq:G2}
\dot{\widehat{x}}(t) = {V_k^T A V_k} \widehat{x}(t) + V_k^T A \overline{x}+ 
V_k^T U_r U_r^T\,\f\left(\overline{x}+V_k \widehat{x}(t)\right) +
V_k^T (\Id_m - U_r U_r^T)\,\f\left(\overline{x}+V_k \widehat{x}(t)\right),
\end{equation}
where $\Id_m\in\R^{m\times m}$ denotes the identity, and the last term (the POD error, as seen from $\mathcal{V}_k$) can be neglected.
However, this still does not solve the problem of computational complexity because it requires all $m$ components of $\f\left( \overline{x}+V_k \widehat{x}(t)\right)$, and  the matrix vector product
$U_r^T\f\left( \overline{x}+V_k \widehat{x}(t)\right)$ takes $\mathcal{O}(mr)$ operations for every time point $t=t_i$. 

The DEIM \cite{DEIM} trick is to select a submatrix of the $m\times m$ identity $\Id_m$, 
$$\WSO\equiv\begin{pmatrix} \Id_m(:,i_1)& \cdots &\Id_m(:,i_r)\end{pmatrix}\in\R^{m\times r},$$
and to replace the orthogonal projector $U_r U_r^T$ by
the oblique projector 
$$
\D \equiv U_r (\WSO^T U_r)^{-1}\WSO^T.
$$
{Note that $\D$ has an interpolating property at the $r$ selected coordinates, 
$\WSO^T\D f = \WSO^T f$.}
The alternative for (\ref{eq:G2}) is thus
\begin{equation}\label{eq:G3}
\dot{\widehat{x}}(t) \approx {V_k^T A V_k} \widehat{x}(t) + V_k^T A \overline{x}+
V_k^T \D\,\f\left( \overline{x}+V_k \widehat{x}(t)\right) 
\end{equation}
where in the matrix product $V_k^T\D$, the factor $V_k^T U_r (\WSO^T U_r)^{-1}$ can be pre-computed in the off-line phase.
Obviously, important is only the component of the error 
$(\Id_m - \D)\f(\overline{x}+V_k \widehat{x}(t))$ that lies in $\mathcal{V}_k$. 

The on-line computation $\WSO^T \f(\overline{x}+V_k \widehat{x}(t))$ at any particular $t$ involves only $r$ values $\f_{i_j}(\overline{x}+V_k \widehat{x}(t))$, $j=1,\ldots, r$.
If $\f$ is defined at a vector $x=(x_i)_{i=1}^m$ component-wise as\footnote{For a general nonlinear $\f(x)=( \varphi_1(x_{\mathcal{I}_1}), \varphi_2(x_{\mathcal{I}_2}), \ldots, \varphi_m(x_{\mathcal{I}_m}))^T$, where $x_{\mathcal{I}_j}$ ($\mathcal{I}_j\subseteq \{ 1, \ldots, m \}$) denotes a sub-array of $x$ needed to evaluate $\varphi_j(x)$, the situation is more complicated, see \cite[\S 3.5]{DEIM}.} $\f(x) = ( \phi_1(x_1), \phi_2(x_2), \ldots, \phi_m(x_m))^T$ then
$$
\WSO^T \f(\overline{x}+V_k \widehat{x}(t)) = \begin{pmatrix} \phi_{i_1}(\overline{x}_{i_1}+V_k(i_1,:) \widehat{x}(t)) \cr  
\phi_{i_2}(\overline{x}_{i_2}+V_k(i_2,:) \widehat{x}(t)) \cr \vdots \cr 
\phi_{i_r}(\overline{x}_{i_r}+V_k(i_r,:) \widehat{x}(t))\end{pmatrix} \equiv
\f_{\WSO}(\WSO^T \overline{x}+(\WSO^T V_k)\widehat{x}(t)) , \;\; t= t_1, t_2, \ldots
$$
and the computational complexity of 
$$
V_k^T\D \f(\overline{x}+V_k \widehat{x}(t)) = (V_k^T U_r)(\WSO^T U_r)^{-1} \f_{\WSO}(\WSO^T \overline{x}+(\WSO^T V_k)\widehat{x}(t))
$$
becomes independent of the dimension $m$, once the time independent matrices are precomputed in the off-line phase.\footnote{In the sequel, for the sake of simplicity, we do not include centering of the snapshots.} This tremendously reduces both the flop count and the memory traffic in the (on-line) simulation.


The error of the  DEIM oblique projection can be bounded in the Euclidean norm by that of the orthogonal projector,
\begin{eqnarray}\label{e_projcond}
\| f - \D f\|_2\leq \kappa\, \|(\Id_m - U_r U_r^T)f\|_2, \quad \text{where}\quad
\kappa\equiv \|(\WSO^T U_r)^{-1} \|_2.
\end{eqnarray}
The condition number $\kappa$ determines the quality of the approximation, and satisfies
$\kappa \leq \mathcal{O}\left(m^{(r-1)/2}\right) / \|U_r(:,1)\|_{\infty}$ \cite{DEIM}.
In practical situations, however, this bound is pessimistic  and $\kappa$ is much lower. (Using the concept of maximal volume \cite{Knuth-volume}, \cite{Goreinov19971}, it can be shown that there exists a strategy such that $\kappa\leq \sqrt{1+r(m-r)}$.)

\subsubsection{{Variations and generalizations}}
{
DEIM has been successfully deployed in many applications, and tuned for better performance, giving rise to 
 the localized DEIM \cite{peherstorfer2014localized}, unassembled DEIM (UDEIM) \cite{Tiso-Rixen-2013}, \cite{TISO:2013:UDEIM},  matrix DEIM \cite{Wirtz-Sorensen-Haasdonk=2014}, \cite{Negri:2015:MDEIM}, nonnegative DEIM (NNDEIM) \cite{NNDEIM-2016}, and
Q-DEIM \cite{drmac-gugercin-DEIM-2016}. The latter is an orthogonal variant of DEIM, which can be efficiently implemented with high-performance libraries such as LAPACK \cite{LAPACK} and ScaLAPACK \cite{ScaLAPACK}. Furthermore, Q-DEIM 
admits a better condition number bound,
$\kappa \leq \sqrt{m-r+1}\mathcal{O}(2^r)$; it allows randomized sampling; and it can work with only a subset of the rows of $U_r$ for computing selection matrices $\WSO$  while keeping $\kappa$ moderate. }
\subsection{GEIM and PBDW}\label{SS=GEIM-PBDW}
{ 
In many applications, the functions' values
may not be available through point evaluation because, e.g., they are from a class that does not contain continuous functions,  there is no analytical expression, or they may be noisy sensor data (measurements) obtained by weighted averaging.} In those cases, point-wise interpolation may not be possible, nor even desirable -- for a most illuminating discussion see \cite{GEIM}. 
{This motivated the development of a generalization of EIM, GEIM (Generalized Empirical Interpolation Method), which replaces point interpolation by more general evaluation functionals selected from a dictionary; see \cite[Chapters 4, 5]{Mula-Thesis} and \cite{maday:hal-00812913}, \cite{GEIM},
\cite{Maday-Mula-Turinici-GEIM-2016} 
}

{These ideas have been further extended in the Parametrized-Background Data-Weak approach to data assimilation (PBDW) \cite{NME:NME4747}. PBDW is an {elaborate data assimilation scheme} whose weak formulation naturally fits variational framework for (parametrized) PDEs, and facilitates error estimates (both \emph{a priori} and \emph{a posteriori}) with the capability to  identify optimal observation functionals.
Additional insights and analysis of PBDW with respect to noise in the data, and an updating strategy for many-query scenarios, are provided in \cite{PBDW-more}; the optimality of the approximation is established in \cite{Binev-data-assim-2017}. {Furthermore, \cite{Binev-data-assim-2017} contains a multi-space extension.}
}

{
In the context of empirical interpolation, PBDW allows more (generalized) approximation positions than the cardinality of the POD basis, thus calling for least squares approximation. In particular, it contains GEIM as a special (one-space) case. }

\subsection{Proper inner-product space structure}\label{SS=UPISPS}

 {In applications in engineering and applied sciences the solution of (\ref{eq:LTI}) represents an approximation of a function from an appropriate function space, that is subject to governing equations that describe a physical reality. 
 	The quality of an approximation is then naturally measured in an appropriate (interpretable) metric in that space.
 	
 For instance, in many applications the natural ambient space is (weighted) $L^2(\Omega)$,
 $
 L^2(\Omega) = \{ f:\Omega\longrightarrow \R \; :\; \int_\Omega|f(x)|^2 \rho(x) dx < \infty\},
 $
 with the Hilbert space structure generated by the inner product 
 $
 	(f,g)_{L^2(\Omega)} = \int_\Omega f(x) g(x) \rho(x) dx,
 $
 and with the corresponding induced norm $\|f\|_{L^2(\Omega)}=\sqrt{(f,f)_{L^2(\Omega)}}$. 
 Both the weight function $\rho(\cdot)$ and a quadrature formula in the course of constructing a discrete (finite $m$-dimensional) framework yield a weighted inner product in $\R^m$, $(u,v)_W = v^T W u$, where $W$ is the corresponding symmetric positive definite matrix. Then the natural framework for devising e.g. a POD approximation \cite[\S 1.2]{volkwein-2011-mor} is given by the Hilbert space structure of $(\cdot,\cdot)_W$. 
 Further, for the equations of e.g. compressible fluid flow, Galerkin projection in an $(\cdot,\cdot)_{L^2(\Omega)}$ inner product may not preserve the underlying physics, such as energy conservation or stability, see e.g. \cite{Rowley2004115}, \cite[\S 3.4.3]{holmes2014turbulence}. Different inner products (with corresponding norms) may yield substantially different results, see e.g. \cite{POD-Sound-Freund}, \cite{POD-Symm-TW}, \cite{Kalashnikova-Arun-2014}. In model order reduction, for instance, a Galerkin projection may be naturally defined in a Lyapunov inner product, generated by the positive definite solution $W$ of a Lyapunov matrix equation, see e.g. \cite[\S 6.1]{ROM-SANDIA-2014}, \cite{Serre20125176}, \cite{Rowley-MRF-2005}, \cite[\S 5.4.3]{holmes2014turbulence}. For further examples and in-depth discussion see \cite{Barone20091932} \cite{Kalashnikova2014569}, \cite{Calo2014204}, \cite{ZIMMERMANN2010165, zimmermann-sisc-2016}, \cite{Satish-et-all-Orr-Sommerfeld}, \cite{Noack-thermo-unsteady}, \cite{NME:NME4820}. {It should be clear that the use of a weighted inner product in the POD-DEIM framework does not guarantee the stability of the reduced system, unless the DEIM is additionally adapted to a particular structure. An excellent example of energy stable DEIM approximation is the NNDEIM \cite{NNDEIM-2016}.}

{The use of a proper inner product is implicitly assumed in the abstract framework of PBDW, including the special case of GEIM. The resulting numerical realization of the proper inner product results in the discrete  $(\cdot,\cdot)_W$ inner product.} 
	From the numerical point of view, this is not a mere change to another inner product, as the condition number of $W$ becomes an important factor both in the theoretical projection error bound and in the computation in finite precision arithmetic. 	
Hence, it seems natural and important to revise the numerical implementation of DEIM oblique projection, to place it in the wider context of PBDW, and to ensure its robustness independent of the possibly high condition number of the weight matrix $W$.}

\subsection{Scaling of variables}\label{SSS-scaling-of-variables}
We discuss difficulties due to scaling issues in the practical computation of a POD basis and 
construction of a DEIM projection, and argue that, when appropriate, 
the DEIM projection must be weighted in a consistent manner with the POD basis.

Scaling issues discussed here arise from two sources.
First, when unknowns $x_i(t)$ represent different physical quantities,
such as  velocity and pressure, and the numerical values of one of them, say pressure, can dominate all others by 
several orders of magnitude.
Second, a single variable can vary over a wide range. 
In both scenarios, the components of $\f(x(t))$ may vary over several orders of magnitude, so that
the matrix of nonlinear snapshots 
$$F\equiv\begin{pmatrix} \f(x(t_1)) & \cdots & \f(x(t_n))\end{pmatrix}\in\R^{m\times n}$$
has graded rows, with widely varying norms. 

Let us try to understand the computational ramifications. Suppose the rows of 
$F=\left(\begin{smallmatrix} B \cr s\end{smallmatrix}\right)$ are permuted so that
$B$ contains the rows with large norm, $s$ the rows with small norm, so that in the Frobenius norm $\|B\|_F \gg \|s\|_F$. 
Typically $m\gg n$, and let the thin SVD be $F = U \Sigma K^T$, where $U$ is $m\times n$ orthonormal, $\Sigma$ is diagonal and  $K$ is an orthogonal matrix. 
An economical way to compute $U_r$, often used in practice, is to first compute the eigenvalue
decomposition $G\equiv F^T F=K \Sigma^2 K^T$. Then choose a suitable $r$, 
compute $U_r = F K(:,1:r)\Sigma(1:r,1:r)^{-1}$, and apply a Gram-Schmidt correction to
improve the numerical orthonormality of $U_r$. 
Since $F^T F = B^T B + s^T s \approx B^T B$, in this procedure the contribution of $s$ to the 
computation of $K$ and $U_r$ is marginal  
and the subdominant variables are almost invisible.

Further, the POD basis may inherit the graded structure of $F$.  Assume, for the purpose of demonstration, 
that the dominant $r$ singular values
of $F$ are nearly equal and much larger than the remaining,  subdominant, singular values.
Rearranging the thin SVD $F = U \Sigma K^T$, where $K$ is an orthogonal matrix, shows 
that the row norms of $U\Sigma = F K$ are the same as those of $F$. Furthermore, 
the row norms of the matrix of the leading $r$ singular vectors $U_r$
are distributed like the corresponding row norms of $F$.  
 The indices corresponding to dominant variables have dominant rows in $U_r$, which creates
 difficulties for the representation of subdominant variables. 
 
 Moreover, the DEIM \cite{DEIM} and Q-DEIM \cite{drmac-gugercin-DEIM-2016} are based on
 greedy algorithms that try to identify an $r\times r$ submatrix of $U_r$ of maximal volume,
 thus preferring row indices corresponding to dominant variables  and ignoring the others. The 
 resulting small
 approximation error in the Euclidean norm is misleading, though. Without prior scaling,
 relevant and informative subdominant variables are unnecessarily suppressed.\footnote{Recall the discussion in \S \ref{SS=UPISPS}.} 
 
Finally, it should be pointed out that strongly graded $F$ poses intrinsic computational difficulties 
for any algorithm for computing $U_r$. Even the backward error $\delta F$ that corresponds to the numerical 
computation, and which is small in the sense that $\|\delta F\|_F/\|F\|_F$ is small, may wipe out the information 
on the subdominant variables. The corresponding entries of the left singular vectors $u_{k}$, $k=1,\ldots, r$, 
are computed  possibly with large relative error, as numerical methods in general compute the singular vectors 
with error such that $\|\delta u_k\|_2$ is appropriately bounded by the machine roundoff times a condition 
number \cite{Wedin1972}, \cite[V.4]{ste-sun-90}. Tiny components of $u_k$ are usually computed with large relative error.

\subsection{Contributions and overview of the paper}
{Our contributions to the theory and practice of empirical interpolation methods in the framework of PBDW approximations (in particular, EIM and GEIM) are towards numerical linear algebra and matrix theory; the goal is to setup a more general algorithmic schemes and principles for development of numerical methods with sharp error bounds, and for their successful software implementations and applications in scientific computing.} 

In Section~\ref{S:SRRQR}, we present a substantial improvement of the bound on the condition number $\kappa$ in (\ref{e_projcond}).
The selection operator $\WSO$ is based on local maximal volume approach \cite{Knuth-volume,Goreinov19971},
implemented via a strong rank-revealing QR decomposition \cite{GuE96}; the resulting DEIM condition number is $\kappa\leq \sqrt{1+\eta^2\, r(m-r)}$, with tunable parameter $\eta\geq 1$.
In \S \ref{S=Canonical}, we present a canonical form for the DEIM projector, 
which is based on the well known structure of oblique projections. It
provides better understanding of the structure of DEIM and its approximation error. 
In \S \ref{S=WDEIM} we introduce and give detailed analysis of the weighted DEIM ($W$-DEIM) which naturally applies in the 
situations discussed in \S \ref{SS=UPISPS}, \S \ref{SSS-scaling-of-variables}. The goal is to establish a universal framework for DEIM projections in weighted inner product spaces, 
where inner products are induced by positive definite matrices $W$ of various origins and with 
various interpretations. 

{We present several algorithms for computing the $W$-DEIM approximation. The algorithms come in two flavors depending on whether generalized or pointwise interpolation is desired. When generalized interpolation is to be used, we present different algorithms depending on whether $W$ is dense or sparse. When pointwise interpolation is used, our analysis shows that the condition number of $W$ plays a role in the error analysis. To mitigate this issue, we present an algorithm for which the spectral condition number $\kappa_2(W)=\|W\|_2 \|W^{-1}\|_2$ enters the error bound, up to the factor of $\sqrt{m}$, as $\sqrt{\min_{D=\mathrm{diag}}\kappa_2(DWD)}$.}

{$W$-DEIM can be considered as a numerical realization of disretized (one-space) PBDW that includes, as a special case, a discrete version 
of the generalized empirical interpolation method (GEIM).}
In \S \ref{S=Examples}, we corroborate 
the results with numerical examples.


\section{Nearly optimal subset selection}\label{S:SRRQR}
We review strong rank revealing QR methods for matrices with at least as many rows as columns
(\S \ref{s_tall}); and present an extension to matrices with fewer rows than columns
and apply it to matrices with orthonormal rows (\S \ref{s_wide}). This yields a new DEIM selection
with superior error bound.

\subsection{Tall and skinny matrices}\label{s_tall}
For $\mat{A}\in\R^{m \times n}$ with $m\geq n$,
and a target rank $r<n$, a QR factorization with column pivoting computes
\begin{eqnarray*}\label{eqn:rrqr}
\mat{A}\mat{\Pi} \quad =\mat{Q} \begin{pmatrix}
\mat{R}_{11} & \mat{R}_{12} \\  0 & \mat{R}_{22} \end{pmatrix},
\end{eqnarray*}
where $\mat{\Pi}\in\R^{n\times n}$ is a permutation;
$\mat{Q} \in \R^{m\times m}$ is an orthogonal matrix; 
$\mat{R}_{11} \in \R^{r\times r}$ and $\mat{R}_{22} \in \R^{(m-r) \times (n-r)}$ are   
upper triangular; and $\mat{R}_{12} \in \R^{r \times (n-r)}$. 

Let $\sigma_1(A)\geq \cdots \geq \sigma_n(A)\geq 0$ be the
singular values of $A$.
Singular value interlacing \cite[Corollary 8.6.2]{GovL13}
implies for the non-increasingly ordered singular values $\sigma_j(\mat{R}_{11})$ and $\sigma_j(\mat{R}_{22})$ of the diagonal blocks $\mat{R}_{11}$ and $\mat{R}_{22}$, respectively,
\begin{eqnarray*}
\sigma_j(\mat{R}_{11}) &\leq & \sigma_j(\mat{A}) ,\qquad \;\;\;1\leq j \leq r\\ 
\sigma_{r+j}(\mat{A}) &\leq  & \sigma_{j}(\mat{R}_{22}), \qquad 1\leq j\leq n-r.
\end{eqnarray*}

So-called \textit{rank-revealing QR (RRQR) factorizations} \cite{ChI91a,GuE96}
try to make the singular values of $\mat{R}_{11}$ 
as  large as possible, and those of $\mat{R}_{22}$ as small as possible.
In particular, the \textit{strong RRQR (sRRQR) factorization}
\cite[Algorithm 4]{GuE96} with tuning parameter $\eta \geq 1$ computes
a triangular matrix $R$ whose diagonal blocks have singular values 
within, essentially, a polynomial factor (in $n$ and $r$) of the singular values of $\mat{A}$,
\begin{eqnarray*}\label{eqn:rrqr_A}
\frac{\sigma_j(\mat{A})}{\sqrt{1 + \eta^2r(n-r)}} &\leq &\sigma_j(\mat{R}_{11}) ,\qquad 1\leq j \leq r\\ 
 \sigma_{j} (\mat{R}_{22})  & \leq &  \sqrt{1 + \eta^2r(n-r)} \,\sigma_{r+j}(\mat{A}), \qquad
 1\leq j\leq n-r ,\nonumber\\
\end{eqnarray*}
and whose off-diagonal block is bounded by
\begin{eqnarray*} \label{eqn:rrqr2}
\left| \left(\mat{R}_{11}^{-1} \mat{R}_{12} \right)_{ij}\right| \leq \eta, \qquad 
1\leq i\leq r, \, 1\leq j\leq n-r.
\end{eqnarray*}
For $\eta>1$  the sRRQR factorization can be computed in
$\mathcal{O}\left((m+n\log_{\eta}{n})n^2\right)$ arithmetic operations \cite[Section 4.4]{GuE96}.
Recommended values for $\eta$ are small fractional powers of $n$ \cite[Section 4.4]{GuE96}, such as 
$\eta=10\sqrt{n}$ \cite[Section 6]{GuE96}, which result in a $\mathcal{O}(mn^2)$ time complexity.

The traditional Businger-Golub QR with column pivoting \cite{bus-gol-65},
\cite[Algorithm 5.4.1]{GovL13}  often achieves the above 
bounds in practice, but fails spectacularly on contrived examples such as the Kahan matrix \cite{kahan-66}, \cite[Section 6]{GuE96}. Sometimes, the failure is caused by the software implementation of a RRQR factorization; for details see \cite{drmac-bujanovic-2008}.

\subsection{Short and fat matrices}\label{s_wide}
The sRRQR factorization can be adapted to matrices with fewer rows than columns, 
$m<n$, and of full row rank, to select a well-conditioned $m\times m$ submatrix.

For $\mat{A}\in\R^{m \times n}$ with $m\leq n$,
and target rank $r=m$, a QR factorization with column pivoting computes
\begin{eqnarray*}
\mat{A}\mat{\Pi}  =\mat{Q} \begin{pmatrix}\mat{R}_{11} & \mat{R}_{12} \end{pmatrix},
\end{eqnarray*}
where $\mat{\Pi}\in\R^{n\times n}$ is a permutation matrix,
$\mat{Q} \in \R^{m\times m}$ is an orthogonal matrix, 
$\mat{R}_{11} \in \R^{m\times m}$ is   
upper triangular, and $\mat{R}_{12} \in \R^{m \times (n-m)}$. 

A sRRQR factorization, in particular, is computed with a simplified version 
of \cite[Algorithm 4]{GuE96}. A column of $R_{11}$ is swapped with 
one in $R_{12}$ until $\left| \left(\mat{R}_{11}^{-1}\mat{R}_{12}\right)_{i,j}\right|\leq \eta$,
$1\leq i \leq m, 1 \leq j \leq n-m$. 
From \cite[Lemma 3.1]{broadbent2010subset} follows
\begin{eqnarray} \label{e_sigmaA}
\frac{\sigma_j(\mat{A})}{\sqrt{1 + \eta^2m(n-m)}} &\leq &\sigma_j(\mat{R}_{11}) ,\qquad 1\leq j \leq m.
\end{eqnarray}
Given a matrix $V$ with $r$ orthonormal columns, this algorithm can be used to select a well conditioned $r\times r$ submatrix.
 
\begin{lemma}\label{l_det}
Let $\mat{V} \in \R^{m\times r}$ with $\mat{V}^T\mat{V}=\mat{I}_r$.
Applying \cite[Algorithm 4]{GuE96} with target rank $r$ and tuning parameter $\eta\geq 1$
to $V^T$ gives a submatrix $\WSO\in\R^{m\times r}$ of $\Id_m$ with
$$\frac{1}{\sqrt{1 + \eta^2r(m-r)} }  \leq \sigma_j(\WSO^T\mat{V}) \leq 1, 
\qquad 1\leq j\leq r ,$$
and
$$ 1   \leq   \|(\WSO^T\mat{V})^{-1}\|_2 \leq   \sqrt{ 1+\eta^2 r(m-r)}.$$
\end{lemma}

\begin{proof} 
Applying \cite[Algorithm 4]{GuE96} to $V^T$ gives
$$\mat{V}^T \begin{pmatrix}\mat{\Pi}_1 & \mat{\Pi}_2\end{pmatrix}  =
\mat{Q}\begin{pmatrix}\mat{R}_{11} & \mat{R}_{12} \end{pmatrix},$$
where $Q\in\R^{r\times r}$ is an orthogonal matrix; $R_{11}\in\R^{r\times r}$
is upper triangular; and $\begin{pmatrix} \Pi_1& \Pi_2\end{pmatrix}\in\R^{m\times m}$
is a permutation matrix with $\Pi_1\in\R^{m\times r}$.

Since $\mat{V}$ has $r$ orthonormal columns,  $\sigma_j(V)=1$, $1\leq j\leq r$.
From (\ref{e_sigmaA}) follows
$$\frac{1}{\sqrt{1 + \eta^2r(m-r)} }  \leq \sigma_j(R_{11}) \leq 1, \qquad 1\leq j\leq r.$$
Set $\WSO=\Pi_1$, so the first block column  equals
$\mat{V}^T\WSO =   \mat{V}^T\mat{\Pi}_1 = \mat{Q}\mat{R}_{11}$.
Since $\mat{Q}$ is an orthogonal matrix, $\mat{V}^T\WSO$ has the same
singular values as $\mat{R}_{11}$. 
\end{proof}

Lemma~\ref{l_det}, applied with $V=U_r$, implies a tremendous improvement for the error of the oblique projector
$\D$ in (\ref{e_projcond}). If $\WSO$ is computed from a sRRQR factorization of the transposed POD basis $U_r^T$, the condition number is bounded by
\begin{eqnarray}\label{e_projcond1}
\kappa\leq  \sqrt{ 1+\eta^2 r(m-r)}.
\end{eqnarray}

\section{Canonical representation of $\D$}\label{S=Canonical}
The DEIM operator is an oblique projection and, as such, it possesses certain canonical structure that is revealed in an appropriately chosen basis. In this section we derive representation of the DEIM projection operator in a particular basis, in order to gain better understanding of the effectiveness of DEIM. {As already mentioned in \S \ref{SS=GEIM-PBDW}, the PBDW framework \cite{NME:NME4747} allows selecting $s\geq r$ approximation points, and we will proceed with the general case of rectangular $\WSO^T U_r$.} {The oversampling has been successfully used in the related context of missing point estimation, see \cite{Astrid-etal-MPE-2008} \cite{Peherstorfer-Willcox-adeim}, \cite{Geom-subspace-Zimm-Per-Will-2015},  \cite{zimmerman-willcox-sisc-2016}.}


We adopt the following notation. Let $\WSO \in \mathbb{R}^{m\times s}$ be a selection of $s$ columns of the identity $\Id_m$ and let $U_r \in \mathbb{R}^{m\times r}$.  Define the orthogonal projectors $\Prj_{\WSO}=\WSO \WSO^T$ and $\Prj_{U_r}=U_r U_r^T$  
onto $\mathcal{R}(\WSO)$ and $\mathcal{R}(U_r)$, respectively. 

\subsection{Generalization of oblique DEIM}\label{SS=Gen-DEIM-3.1} 
We first derive a representation of the DEIM projection in terms of $\Prj_{\WSO}$ and $\Prj_{U_r}$. {Suppose $\WSO$ and $U_r$ have full column rank,
then  the DEIM projector $\D$ can be written as~\cite[Theorem 2.2.3]{Bjo15}}
\begin{equation}\label{eq:D1}
\D = U_r (\WSO^T U_r)^{-1}\WSO^T  = (\Prj_{\WSO} \Prj_{U_r})^{\dagger},
\end{equation}
where the superscript $\dagger$ denotes the Moore-Penrose inverse. Note that the expression $(\Prj_{\WSO} \Prj_{U_r})^{\dagger}$ does not require existence of the inverse $(\WSO^T U_r)^{-1}$; in fact it does not even require $\WSO$ and $U_r$ to have the same number of columns, or the same rank.

We now consider the case that $\WSO^TU_r \in \mathbb{R}^{s \times r}$ is a rectangular matrix where $s\neq r$. 
In this case, one can check (e.g., using the SVD of $\WSO^T U_r$) that it holds
\begin{equation}
(\Prj_{\WSO} \Prj_{U_r})^{\dagger} = U_r (\WSO^T U_r)^{\dagger}\WSO^T.
\end{equation}
This observation leads to a general definition of the  DEIM projection as $\D = (\Prj_{\WSO} \Prj_{U_r})^{\dagger}$ which is valid when $\WSO$ has different number of columns as $U_r$, and different rank. We now investigate whether this generalization retains the properties of interpolation ($\WSO^T \D f=\WSO^T f$) and projection ( $\D \Prj_{U_r} = \Prj_{U_r}$).

With the observation $\rank(\D) = \rank(\WSO^T U_r) = \min\{ s, r\}$, suppose that $s\neq r$ and  split the analysis into two cases. 
\smallskip

\begin{enumerate}
\item Case $\rank (\D) = s < r$  

\begin{enumerate}
\item The interpolation property still holds, i.e., 
$$
\WSO^T (\D f ) = \WSO^T U_r (\WSO^T U_r)^{\dagger}\WSO^T f = \WSO^T f .
$$ 
The reason for this is because $\WSO^T U_r$ has full row rank, and $(\WSO^T U_r)^{\dagger}$ is a right multiplicative inverse. 
\item On the other hand, the projection property is lost, i.e.,  $\D \Prj_{U_r} \neq \Prj_{U_r}$. However, $\D$ is still a projector, $\D^2=\D$. To find its range, let $W_s$ be the leading right $s$ singular vectors of $\WSO^T U_r$. Then the DEIM projection operator. $\D \Prj_{U_r} = \Prj_{V_s}$, where $V_s= U_r W_s$ spans an  $s$--dimensional subspace of $\mathcal{R}(U_r)$. Therefore, $\D$ is a projector onto $\mathcal{R}(V_s) \subset \mathcal{R}(U_r)$.

\end{enumerate}

\item Case $\rank(\D) = r < s $
\begin{enumerate}

\item The interpolation property does not hold, i.e., $\WSO^T (\D f ) \neq \WSO^T f$. This is because $(\WSO^T U_r)^{\dagger}$ is no longer a right multiplicative inverse. However,  $\WSO^T (\D f )$ is the least square projection of $\WSO^T f$ onto the range of $\WSO^T U_r$. To see this
$$
\WSO^T (\D f ) = \WSO^T U_r (\WSO^T U_r)^{\dagger}\WSO^T f =
\Prj_{\mathcal{X}} (\WSO^T f) ,\;\; \mathcal{X}=\mathcal{R}(\WSO^T U_r) .
$$ 

\item In this case $\D \Prj_{U_r} = \Prj_{U_r}$ since $(\WSO^T U_r)^{\dagger}$ is a left multiplicative inverse of $\WSO^T U_r$. 
 
\end{enumerate}
 \end{enumerate}
As can be seen above, when the DEIM operator is generalized to the setting $s \neq r$ only the projection property or the interpolation property is retained but not both simultaneously. For related developments, see~ \cite{NME:NME4747}, \cite{zimmerman-willcox-sisc-2016}, \cite{Casenave-EIM-variants-2016}.

\subsection{Canonical structure of $\D$}
We present the following theorem that sheds light onto the canonical structure of the DEIM operator $\D$. 
\begin{theorem}\label{TM-canonical-form}
{Let $U_r\in\R^{m\times r}$ and $\WSO\in\R^{m\times s}$ have orthonormal columns, and  $\D = U_r (\WSO^T U_r)^{\dagger}\WSO^T$ and assume that $1\leq r,s \leq m$. Let 
$\ell\equiv \mathrm{dim}(\mathcal{R}(\WSO) \bigcap \mathcal{R}(U_r))$,  set
$p\equiv \mathrm{rank}(\D) - \ell$, and 
	let  the singular values $\sigma_i=\cos\psi_i$ of $\WSO^T U_r$ be ordered as
	\begin{equation}
	1=\sigma_1=\ldots = \sigma_\ell > \sigma_{\ell+1}\geq \ldots \geq  \sigma_{\ell+p}>\sigma_{\ell+p+1} = \ldots = \sigma_{\min(r,s)}=0 .
	\end{equation}
	(Here 
	$0<\psi_{\ell+1}\leq\ldots \leq \psi_{\ell+p} <\pi/2$ are 
	the acute principal angles between the ranges of $\WSO$ and $U_r$.)}

	\emph{(i)}	There exists an orthogonal $m\times m$ matrix $Z$ such that the matrix $\D$ 
	can be represented as
	\begin{equation}\label{eq:D:canonical}
	\D =  (\Prj_{\WSO} \Prj_{U_r})^{\dagger} = Z \begin{pmatrix}
	\Id_\ell &   &    \cr
	& {\displaystyle \bigoplus_{i=1}^p T_i} &  \cr 
	&                          & \0
	\end{pmatrix} Z^T,\;\; T_i = \begin{pmatrix}
	1 & 0 \cr \tan\psi_{\ell+i} & 0 
	\end{pmatrix} .
	\end{equation}	
Here the $\0$ block is of size $m-\ell -2p$.
 
	\emph{(ii)} The DEIM projector $\D$ satisfies 
	{$\|\D\|_2 = 1/ \cos\psi_{\ell+p}$}. 
{If, in addition, $\D\neq \0$ and $\D\neq \Id_m$, then $\|\D\|_2 =  \|\Id_m-\D\|_2 =  1/ \cos\psi_{\ell+p} $.} 
\end{theorem}
\begin{proof} The above representation follows immediately from the canonical representation of a pair of orthogonal projectors \cite{wed-82}. In a particularly constructed orthonormal basis given by the columns of $Z$, the two projectors have the following matrix representations:  
	\begin{eqnarray}
	\Prj_{\WSO} &=& Z \left(\begin{smallmatrix}
	\Id_\ell &   &    \cr
	& {\displaystyle \bigoplus_{i=1}^p J_i} &  \cr 
	&                          & D_{s}
	\end{smallmatrix}\right) Z^T,\;\; \mbox{where}\;\; J_i = \begin{pmatrix}
	1  \cr 0  \end{pmatrix} \begin{pmatrix}
	1 & 0 \end{pmatrix},\;\;\mbox{and}  \label{eq:PS}\\ 
	\Prj_{U_r} &=& Z \left(\begin{smallmatrix}
	\Id_\ell &   &    \cr
	& {\displaystyle \bigoplus_{i=1}^p \Psi_i} &  \cr 
	&                          & D_u
	\end{smallmatrix}\right) Z^T,\;\; \Psi_i = \begin{pmatrix}
	\cos\psi_{\ell +i} \cr \sin\psi_{\ell+i}\end{pmatrix} 
	\begin{pmatrix}
	\cos\psi_{\ell+i} & \sin\psi_{\ell+i} \end{pmatrix} , \label{eq:PU}
	\end{eqnarray}		
	with  $\psi_{\ell+i}$'s as stated in the theorem, and $D_s$, $D_{u}$ are diagonal matrices with diagonal entries $0$ or $1$ and such that $D_s D_u=\0$. 
	Note that each $(D_s)_{ii}=1$ ($(D_u)_{ii}=1$) corresponds to a direction in the range of $\WSO$ ($U_r$) orthogonal to the entire range of $U_r$ ($\WSO$).
	In the special case when $\WSO^TU_r$ is invertible,  $D_s=D_u=\0$.	

	The expression for $\D$ is obtained by multiplying the representations in (\ref{eq:PS}) and (\ref{eq:PU}), and taking the pseudoinverse. It follows that 
\[ (\mathcal{P}_\WSO\mathcal{P}_{U_r})^\dagger = Z \left(\begin{smallmatrix}
	\Id_\ell &   &    \cr
	& {\displaystyle \bigoplus_{i=1}^p (J_i\Psi_i)^\dagger} &  \cr 
	&                          & \0
	\end{smallmatrix}\right) Z^T. \]
A direct evaluation shows that 

\[ (J_i \Psi_i)^\dagger = \left[  \begin{pmatrix} 1 \cr 0  \end{pmatrix} \cos\psi_{\ell+i} \begin{pmatrix} \cos\psi_{\ell+i} & \sin\psi_{\ell+i}  \end{pmatrix}\right]^\dagger  = \begin{pmatrix}
	1 & 0 \cr \tan\psi_{\ell+i} & 0 
	\end{pmatrix} = T_i.\]

From the canonical representation~\eqref{eq:D:canonical} each block $T_i$ has the norm 
$$\|T_i \|_2 = {\sqrt{1+\tan^2\psi_{\ell+i}}}={1/\cos\psi_{\ell+i}}.$$ 
Therefore, it also follows that $\| \D \|_2 = 1/\cos\psi_{\ell+p}$. 
From (\ref{eq:D:canonical}) we can also derive the canonical form of $\Id_m-\D$: 
\[ \Id_m - \D = Z \left(\begin{smallmatrix}
	\0 &   &    \cr
	& {\displaystyle \bigoplus_{i=1}^p (\Id_2 - T_i)} &  \cr 
	&                          & \Id
	\end{smallmatrix}\right) Z^T.\]
The $\0$ block has dimensions $\ell$, whereas the identity block has dimensions ${m-\ell -2p}.$ When $\D \neq \0, \Id_m$, from~\cite[Corollary 5.2]{IpsM94}  and~\cite{Szyld2006} it follows that $\|\D\|_2 = \|\Id_m-\D\|_2$. 
\end{proof}

The novelty and the importance of Theorem \ref{TM-canonical-form} are in the interpretation in the DEIM setting, allowing for a  deeper understanding of the structure of the DEIM projection and its error. For related usage of canonical angles between subspaces \cite{bjo-gol-73}, see the construction of the favorable bases in \cite{Binev-data-assim-2017}.

\begin{remark}
	{\em 
Let $\WSO^TU_r$ be invertible and $\D = U_r (\WSO^TU_r)^{-1} \WSO^T$. If $f \in \mathcal{R}(U_r)$ then both the DEIM error and the orthogonal projection error are zero, as $\D f = \Prj_{U_r} f =f$. 
In the case  $f\neq \Prj_{U_r} f$, 
write $ f - \D f = (\Id_m - \Prj_{U_r})f + (\Prj_{U_r}-\D)f$; verify that  the summands are orthogonal, apply Pythagoras' theorem to get 
\[ \| f - \D f \|_2^2 = \|(\Id_m - \Prj_{U_r})f\|_2^2 + \|\D f - \Prj_{U_r} f\|_2^2. \]
Since $f \notin \mathcal{R}(U_r)$, we can factor out $\|(\Id_m - \Prj_{U_r})f\|^2$ to get
\begin{equation}\label{eq:kappa}
\| f - \D f \|_2 = \kappa' \| f - \Prj_{U_r} f\|_2 \qquad \kappa' \equiv \sqrt{1+ \frac{\| \D f - \Prj_{U_r} f\|_2^2}{\|f-\Prj_{U_r} f\|_2^2}} .
\end{equation}
(This is illustrated graphically in Figure~\ref{f_deim}.) 
Next, introduce the partition of $f$, represented in the basis $Z$, as follows:
$$
Z^T f=\left( \begin{smallmatrix} f_{[0]} \cr f_{[1]}\cr\vdots\cr f_{[p]} \cr f_{[p+1]} \end{smallmatrix}\right),\;\;f_{[0]}\in\mathbb{R}^\ell,\;\;f_{[1]},\ldots, f_{[p]}\in\mathbb{R}^2,\;\;f_{[p+1]}\in\mathbb{R}^{m-(\ell+2p)} .
$$
Now, straightforward computation for each $i=1,\ldots, p$ reveals that
\begin{eqnarray*}
	\| ( \Id_2 - \Psi_i) f_{[i]}\|_2 &=& \cos\psi_{\ell+i} \left\| \left( \begin{smallmatrix} 
		\frac{\sin^2\psi_{\ell+i}}{\cos\psi_{\ell+i}}& -\sin\psi_{\ell+i} \cr
		-\sin\psi_{\ell+i} & \cos\psi_{\ell+i}
	\end{smallmatrix}\right) f_{[i]}\right\|_2  \\
	\| ( T_i - \Psi_i) f_{[i]}\|_2 &=& \sin\psi_{\ell+i} \left\| \left( \begin{smallmatrix} \sin\psi_{\ell+i} & -\cos\psi_{\ell+i} \cr \frac{\sin^2\psi_{\ell+i}}{\cos\psi_{\ell+i}} & -\sin\psi_{\ell+i}\end{smallmatrix}\right) f_{[i]}\right\|_2  
	= \tan\psi_{\ell+i} \| ( \Id_2 - \Psi_i) f_{[i]}\|_2. 
\end{eqnarray*}
Together this gives 
\begin{align*}
\| \D f - \Prj_{U_r} f\|_2^2 = \| Z^T\D Z Z^T f - Z^T \Prj_{U_r} Z Z^T f \|_2^2  = & \>  \sum_{i=1}^p\tan^2 \psi_{\ell+i} \| ( \Id_2 - \Psi_i) f_{[i]}\|_2^2,\\
\end{align*}
Since $\WSO^TU_r$ is invertible, from the proof of Theorem~\ref{TM-canonical-form}, we have $D_u = \0$, and therefore 
\[ \|f-\Prj_{U_r} f\|_2^2 =\| Z^T f - Z^T \Prj_{U_r} Z Z^T f \|_2^2 =  \> \sum_{i=1}^p \| ( \Id_2 - \Psi_i) f_{[i]}\|_2^2 + \| f_{[p+1]}\|_2^2. \]
Since $f\notin \mathcal{R}(U_r)$, we can divide throughout by $\|f-\Prj_{U_r} f\|_2^2$ to obtain the inequality 
\[\sum_{i=1}^p \frac{\| ( \Id_2 - \Psi_i) f_{[i]}\|_2^2}{\|f-\Prj_{U_r} f\|_2^2} \leq 1. \]
Combining this inequality with the relation for $\| \D f - \Prj_{U_r} f\|_2^2$ into~\eqref{eq:kappa} gives
\begin{eqnarray*}
\frac{\|f-\D f\|_2^2}{\|f-\Prj_{U_r} f\|_2^2} = & \>  1 + \sum_{i=1}^p \tan^2 \psi_{\ell+i} \frac{\| ( \Id_2 - \Psi_i) f_{[i]}\|_2^2}{\|f-\Prj_{U_r} f\|_2^2} \\
\leq   & \>    1 + \tan^2\psi_{\ell + p} = \frac{1}{\cos^2\psi_{\ell+p}} .
\end{eqnarray*}
Therefore, $\kappa' \leq \| \D\|_2 = 1/\cos\psi_{\ell+p}$. 
This result, of course, reproduces the bound~\eqref{e_projcond}. However, the analysis shows that a tighter condition number $\kappa'$ can be obtained by considering how the contributions of the error are weighted in the principal directions identified in Theorem \ref{TM-canonical-form}. 
}
\end{remark}

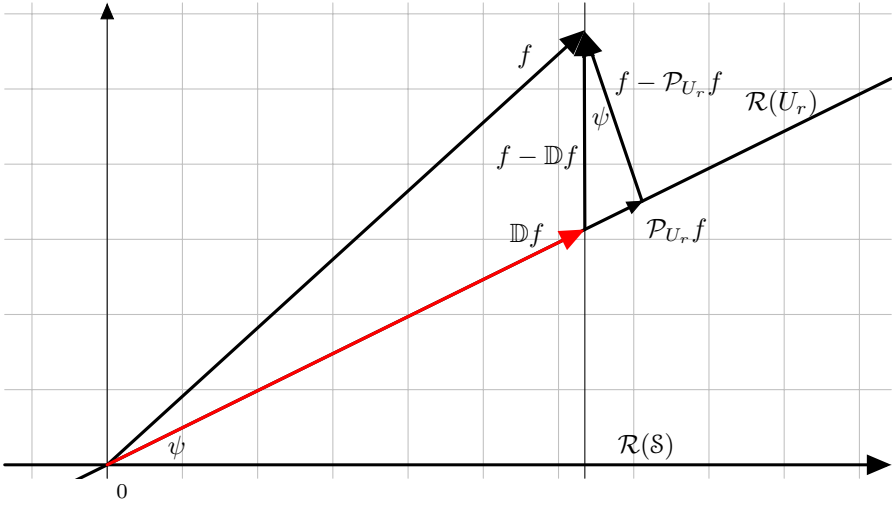
\begin{figure}
	\begin{center} $\quad$
\definecolor{ffqqqq}{rgb}{1.,0.,0.}
\definecolor{cqcqcq}{rgb}{0.7529411764705882,0.7529411764705882,0.7529411764705882}
\begin{tikzpicture}[line cap=round,line join=round,>=triangle 45,x=0.5cm,y=0.6cm]
\draw [color=cqcqcq,, xstep=1.0cm,ystep=1.0cm] (-2.7252425228757273,-0.2986983065545068) grid (20.847851205115493,10.243110234731438);
\draw[line width=0.4mm,->,color=black] (-2.7252425228757273,0.) -- (20.847851205115493,0.);
\draw[->,color=black] (0.,-0.2986983065545068) -- (0.,10.243110234731438);
\draw[color=black] (0pt,-10pt) node[right] {\footnotesize $0$};
\clip(-2.7252425228757273,-0.2986983065545068) rectangle (20.847851205115493,10.243110234731438);
\draw [line width=1.2pt,domain=-2.7252425228757273:20.847851205115493] plot(\x,{(-0.--10.1*\x)/24.58});
\draw [->,line width=1.2pt] (0.,0.) -- (12.7,9.64);
\draw (12.7,-0.2986983065545068) -- (12.7,10.243110234731438);
\draw [->] (0.,0.) -- (14.254377314544714,5.857168872127812);
\draw [->,line width=1.2pt,color=ffqqqq] (0.,0.) -- (12.7,5.21847030105777);
\draw [->,line width=1.2pt] (14.222584761506107,5.84410521119657) -- (12.680291891150095,9.625040459109206);
\draw (10.65492101045877,9.521239047724844) node[anchor=north west] {$f$};
\draw (16.726882497165928,8.504978146085737) node[anchor=north west] {$\mathcal{R}(U_r)$};
\draw (13.32636286403272,0.89828077055530764) node[anchor=north west] {$\mathcal{R}(\WSO)$};
\draw (10.45837249235593,5.5513790796014625) node[anchor=north west] {$\D f$};
\draw (14.09262749476093,5.662962013724084) node[anchor=north west] {$\Prj_{U_r}f$};
\draw (13.244223327426207,8.915697959382592) node[anchor=north west] {$f-\Prj_{U_r}f$};
\draw (10.087190455348585,7.280949157014725) node[anchor=north west] {$f-\D f$};
\draw (1.3669082220811837,0.84849655076174946) node[anchor=north west] {$\psi$};
\draw (12.636457493157596,8.152173003401993) node[anchor=north west] {$\psi$};
\draw [->,line width=1.2pt] (12.7,5.21847030105777) -- (12.683494938061466,9.617188180725055);
\end{tikzpicture}
	\end{center}
	\caption{(Cf. \cite[Figure 1]{GEIM}) DEIM interpolatory projection and its comparison with the corresponding orthogonal projection. Even in the general $m$-dimensional case, the nontrivial action of DEIM projection consists of $\mathrm{dim}(\mathcal{R}(\WSO) \bigcap \mathcal{R}(U_r))$--dimensional identity and $\mathrm{rank}(\D)-\mathrm{dim}(\mathcal{R}(\WSO) \bigcap \mathcal{R}(U_r))$ $2$--dimensional oblique (interpolatory) projections as shown in the figure.}
\label{f_deim}
\end{figure}


\subsection{Connection to CS decomposition}
The structure of $\D$ can also be analyzed using the Cosine--Sine (CS) decomposition~\cite{ste-82}. Assume for simplicity that the rows of $U_r$ are ordered so that $\WSO=\Id_m(:,1:r)$. If this is not the case, we work with $\Pi^T \D \Pi$, where $\Pi$ is a permutation matrix. Assume that $\WSO^TU_r$ is invertible and therefore, the DEIM operator is $\D = U_r (\WSO^T U_r)^{-1}\WSO^T$. Further, let $\WSO_\perp=\Id_m(:,r+1:m)$.

With these assumptions, $U_r$ has the CS decomposition
\[ U_r = \begin{pmatrix} \WSO^TU_r \\ \WSO_\perp^T U_r\end{pmatrix}  = \begin{pmatrix} \Omega_1 & \\ & \Omega_2 \end{pmatrix} \begin{pmatrix} \mathrm{Cos}\Psi \\ \mathrm{Sin}\Psi\end{pmatrix}\Gamma^T. \]
Here $\Omega_1, \Gamma\in \mathbb{R}^{r\times r}$ and $\Omega_2 \in \mathbb{R}^{(m-r)\times (m-r)}$ are orthogonal matrices and 
\[ \mathrm{Cos}\Psi =\mathrm{diag}(\cos\psi_i)_{i=1}^r \in \mathbb{R}^{r\times r}, \qquad \mathrm{Sin}\Psi =\mathrm{diag}(\sin\psi_i)_{i=1}^r \in \mathbb{R}^{(m-r)\times r}.\]
We can therefore represent $\D$ as 
\begin{displaymath}
 \D = \begin{pmatrix} \Omega_1\, \mathrm{Cos}\Psi\, \Gamma^T \cr \Omega_2 \,\mathrm{Sin}\Psi\, \Gamma^T \end{pmatrix} \Gamma\, (\mathrm{Cos} \Psi)^{-1}\, \Omega_1^T \begin{pmatrix} \Id _r & \0 \end{pmatrix} 
=  \begin{pmatrix} \Id_r & \0 \cr \mathrm{Tan}\Psi & \0 \end{pmatrix},
\end{displaymath}
where $\mathrm{Tan}\Psi =  \Omega_2 \mathrm{Sin}\Psi(\mathrm{Cos} \Psi)^{-1} \Omega_1^T = \Omega_2 \mathrm{diag}(\tan\psi_i)_{i=1}^r \Omega_1^T$. Similarly, we have
$$
\Id_m - \D = \begin{pmatrix} \0 & \0 \cr - \mathrm{Tan}\Psi & \Id_{m-r} \end{pmatrix},
$$
and we (again) see that 
$\|\D\|_2=\|\Id_m -\D\|_2=\sqrt{1+\|\mathrm{Tan}\Psi\|_2^2}.$
For further insights on the tangents between subspaces, see e.g., \cite{Angles-KnyazevZhu}.
	
\section{Weighted DEIM}\label{S=WDEIM}
{As discussed in \S\ref{SS=UPISPS}, 
	the discrete analogue of a (generalized) interpolatory projection based approximation must be constructed within an appropriate  weighted inner product, and the selection of the interpolation indices must ensure sharp error bounds. In particular, care must be taken to control how the condition number of the positive definite weight matrix $W$ influences the projection error, expressed in the $W$-weighted norm $\|u\|_W=\sqrt{u^TWu}$. In this section, we address this issue and propose two new algorithms for $W$-weighted variants of DEIM.}

{To set the scene and to introduce notation, in \S \ref{SS=Scene-WPOD} we recall the weighted POD.} In \S\ref{s_wdeim}, we propose $W$-DEIM oblique projection that relies on a more general form of the selection operator, and in the numerical realization uses $W$ implicitly through its Cholesky factor. In this case, although the pointwise interpolation is lost, the more general interpolation condition in the sense of GEIM holds true. In \S\ref{ss_pointwise} and \S\ref{ss_pointwise_scaling} we propose alternative methods for point selection in the weighted setting that allow for pointwise interpolation; however, the resulting approximation error bounds depend on the condition number of $W$ or on the condition number of optimally scaled $W$. 

\subsection{Setting the scene}\label{SS=Scene-WPOD}
Let $W\in\R^{m\times m}$ be symmetric positive definite, and define the weighted inner product 
for $u,v\in\R^m$ by 
$(u,v)_W \equiv  v^T W u.$
Let $W = LL^T$ be a factorization where the nonsingular matrix $L$ is
a Cholesky factor or the positive definite square root $L=W^{1/2}$. 
The original problem might give rise to a nonsingular matrix $L$, 
so that the weight matrix $W=LL^T$ is then given implicitly by its factor $L$. 
{Recall that any two square {``Cholesky''} factors of $W$ are related by an orthogonal matrix $Q$, so that
	$W^{1/2}=LQ$ \cite[page 67, {Exercise} (x)]{IIbook}.} 

\begin{remark}\label{RE:||W}
	{\em 
		In the weighted norm
		$\|u\|_W\equiv \sqrt{(u,u)_W}=\sqrt{u^TWu}=\|L^Tu\|_2,$
		the induced operator norm of an $M\in\R^{m\times m}$ equals
		$$
		\|M\|_W = \max_{x\neq 0}\frac{\|Mx\|_W}{\|x\|_W}= \max_{y\neq 0}\frac{\|L^T M L^{-T}y\|_2}{\|y\|_2}=\|L^{T}ML^{-T}\|_2.$$ 
		Further, in the $W$-inner product space, the adjoint of $M$ is 
		$M^{[T]}\equiv W^{-1} M^T W,$
		where $M^T$ is the transpose of $M$. 	
	}
\end{remark}


The POD basis with respect to $(\cdot,\cdot)_W$  is determined by the 3-step procedure in Algorithm \ref{zd:ALG:POD}. 
For more details see  \cite{volkwein-2011-mor}. For the sake of simplicity, we do not include centering of the snapshots matrix $Y$.
%
%
%
\begin{algorithm}[hbt]
\caption{$\WU=\mathrm{POD}(Y,W\equiv LL^T)$}
\label{zd:ALG:POD}
\begin{algorithmic}[1]
\REQUIRE  Symmetric positive definite $W\in\R^{m\times m}$, or
$L\in\R^{m\times m}$ such that $W=LL^T$ is positive definite.
Matrix $Y\in\R^{m\times n_s}$ of $n_s$ snapshots.
\STATE Compute the thin SVD  $L^T Y = U \Sigma V^T$.
\STATE Determine an appropriate index $1\leq r\leq \rank(L^TY)$ and select ${U}_r\equiv U(:,1:r)$.
\ENSURE $\WU \equiv L^{-T}{U}_r$.
\end{algorithmic}
\end{algorithm}

\noindent Algorithm~\ref{zd:ALG:POD} computes a matrix $\WU$ whose columns are  
$W$-orthonormal, {i.e.,} 
$\WU^T W \WU=\Id_r,$  
and the POD projection in the weighted inner product space is represented by 
\begin{equation}\label{eq:WPU}
\WP_{\WU} \equiv \WU \WU^T W = L^{-T}{U}_r U_r^T L^T.
\end{equation}
Note that $\WP_{\WU}^2=\WP_{\WU}$ and that
$\WP_{\WU}^{[T]}=\WP_{\WU}.$ In fact, $Y = \WU \Sigma V^T$ is a GSVD \cite{van1976generalizing} of $Y$.

\begin{remark}\label{R:<T>}
	{\em 
For $\R^{m\times r}\ni \WU: ( \R^r,(\cdot,\cdot)_2)\longrightarrow (\R^{m},(\cdot,\cdot)_W)$, 
the adjoint matrix in the two inner products is, by definition, given as $\WU^{<T>}=\WU^T W$. 
Hence, $\WU^{<T>}\WU=\Id_r$ and we can write the $W$-orthogonal projector (\ref{eq:WPU}) conveniently in the usual form as 
$\WP_{\WU} = \WU\WU^{<T>}$.	Recalling the discussion from \S \ref{SS=UPISPS}, the projected problem (\ref{eq:G1}) is then computed in the sense of $(\cdot,\cdot)_W$.
}
\end{remark}



\subsection{$W$-DEIM}\label{s_wdeim}
Once a discrete inner product $(\cdot,\cdot)_W$ has been chosen
to capture the geometric framework (e.g., for Petrov-Galerkin projection, POD), 
one needs to define an appropriate DEIM projection operator in this weighted setting. 
Furthermore, the resulting quantities are now measured in the weighted norm 
$\|x\|_W$. To that end, using the notation introduced in Remark \ref{R:<T>},    
we define a $W$-DEIM projector as follows.
\begin{definition}\label{zd:eq:DEF:WDEIM}
Let   $\WU\in\R^{m\times r}$ be $W$-orthogonal. With a full column rank generalized selection operator $\SO\in\R^{m\times s}$ (where $s \geq r$), define
a weighted $W$-DEIM projector 
\begin{equation}
\D \equiv \WU (\SO^{<T>}\WU)^{\dagger}\SO^{<T>} = \WU (\SO^T W \WU)^{\dagger}\SO^T W.
\end{equation}
\end{definition}
In the above definition, in addition to {the use of a} more general inner product, we also allow {for} tall rectangular $\SO^T W\WU {\in \mathbb{C}^{s\times r}}$. The only constraint is that $\SO^T W \WU$ has full column rank. {However, in practice, we will use the  square nonsingular case}.

For the moment, we leave the (generalized) selection operator $\SO$ unspecified, and we remark that the columns of $\SO$ need not be the columns of the identity matrix. \footnote{In fact, one can also allow full row rank to obtain a further variation of the DEIM projection as discussed in \S \ref{SS=Gen-DEIM-3.1}, but  we omit this for the sake of brevity.}
As in the case of DEIM, the matrix $\D$ is an oblique projector, {i.e., it satisfies}
$\D^2=\D.$ 

{The following proposition is a recast of \cite[Proposition 2.1]{zimmerman-willcox-sisc-2016} to the $\|\cdot\|_W$ norm.}
\begin{proposition}
Let $\D$ be {as} in Definition~\ref{zd:eq:DEF:WDEIM} and let $\SO^T W\WU$ have full column rank. Then 
\begin{equation}\label{e_wdeim_inter}
\| f - \D f \|_W \leq  \|\D\|_W \| f - \WP_{\WU} f\|_W .
\end{equation} 
\end{proposition}
\begin{proof} Since  $\SO^T W \WU$ has full column rank, $(\SO^T W \WU)^{\dagger}$ is a left inverse, so that
$\D \WP_{\WU} = \WP_{\WU},$
hence $(\Id_m - \D)\WP_U = 0$. Consequently for any vector $f\in\R^m$
\begin{equation}
(\Id_m - \D) f = (\Id_m-\D)(\Id_m-\WP_{\WU})f.
\end{equation}
Since $\D$ is  non-trivial projector ($\D\neq\0$, $\D\neq\Id_m$) it holds that $\|\D\|_W=\|\Id_m-\D\|_W$, and~\eqref{e_wdeim_inter} follows.
\end{proof}

The condition number that amplifies the POD projection error 
$\| f - \WP_{\WU} f\|_W$ is the weighted norm $\|\D\|_W$. A naive application of the result in Remark~\ref{RE:||W} suggests the bound $\| \D\|_W \leq \sqrt{\kappa(W)} \| \D\|_2$. That is, the condition number of the inner product matrix $W$ could potentially amplify the $W$-DEIM projection error. However, by a clever choice of $\SO$ we can eliminate the factor $\sqrt{\kappa(W)}$.

\begin{definition}\label{d_wso}
{Let the weighted selection operator $\SO$ and the corresponding $W$-DEIM projector $\D$, respectively, be defined as 
\begin{equation}\label{zd:eq:SL}
\SO^T = \WSO^T L^{-1},\;\;\D \equiv \WU (\WSO^T U_r)^{\dagger}\WSO^T L^T
= L^{-T} U_r (\WSO^T U_r)^{\dagger}\WSO^T L^T,
\end{equation}
where $\WSO$ is an $m\times s$ index selection operator ($s$ selected columns of the identity $\Id_m$, $s\geq r$). }
\end{definition}

Note that while $\SO$ is possibly dense, $\WSO$ is a sparse matrix. We now present a result that quantifies the condition number $\|\D\|_W$ for the specific choice of selection operator $\SO$.
\begin{proposition}\label{zd:PROP:SL}
Let $\SO$ and $\D$ be defined as in~\eqref{zd:eq:SL}. Then $\SO^T W \SO=\Id_k$ and $\|D\|_W = \|(\WSO^T U_r)^{\dagger}\|_2$. 
\end{proposition}
\begin{proof} Recall that $L^T\WU = U_r$ and by (\ref{zd:eq:SL}), $\SO^T L=\WSO^T$. Following Remark \ref{RE:||W}, 
straightforward computation yields 
\begin{align} \nonumber
\| \D \|_W = & \>  \| L^T (L^{-T}U_r)(\SO^T LL^T (L^{-T}U_r))^{\dagger}\SO^T LL^T L^{-T}\|_2 \\ 
= & \> \| U_r (\SO^T L U_r)^{\dagger} \SO^T L\|_2 ,
\end{align}	
where, by (\ref{zd:eq:SL}), $\SO^T L=\WSO^T$, and thus $\| \D \|_W =\| U_r (\WSO^T  U_r)^{\dagger} \WSO^T\|_2 = \| (\WSO^T  U_r)^{\dagger}\|_2$.
\end{proof}

Therefore, with this choice of $\SO$, the condition number of $\|W\|_2$ does not explicitly appear in the bounds.  However, the dependence on $W$ is implicitly contained in the matrix $U_r$ of the left singular vectors, and in the definition of $\WU$. 


{In \S\ref{ss_pointwise} we present alternative choices for the Selection Operator $\SO$ which can ensure pointwise interpolation.}

\begin{remark}
{\em 
To obtain the canonical structure of $W$-DEIM, one follows the derivation from \S \ref{S=Canonical}, properly adapted to the structure induced by $(\cdot,\cdot)_W$.
}
\end{remark}
 
\subsection{How to choose $\WSO$}
Recall that $\WSO$ contains carefully chosen columns of $\Id_m$. The index selection to determine the columns of $\WSO$ can be computed using the original DEIM approach \cite{DEIM}. Another approach, Q-DEIM proposed in \cite{drmac-gugercin-DEIM-2016}, uses a rank revealing QR factorization~\cite{bus-gol-65}, {implemented} in high performance software libraries such as LAPACK \cite{LAPACK} and ScaLAPACK \cite{ScaLAPACK}.  

However, in this paper, we adopt the {strong Rank Revealing QR} (sRRQR) factorization \cite[Algorithm 4]{GuE96}. We present a result that characterizes the error of $W$-DEIM
\begin{theorem}\label{t_dgeim_rrqr}
Applying sRRQR~\cite[Algorithm 4]{GuE96} to $U_r$ produces
an index selection operator $\WSO$ whose {\rm $W$-DEIM} projection error satisfies
	\begin{equation}
	\| f - \D f\|_W \leq \sqrt{1+\eta^2 r (m-r)} \| f - \WP_{\WU} f\|_W .
	\end{equation}
\end{theorem}
\begin{proof}
Combining~\eqref{e_wdeim_inter} and Proposition~\ref{zd:PROP:SL} gives  
\[ \| f - \D f\|_W \leq \| (\WSO^TU_r)^{\dagger}\|_2 \|f - \WP_{\WU} f\|_W .\]
Since $U_r$ has orthonormal columns sRRQR~\cite[Algorithm 4]{GuE96} gives a selection operator $\WSO \in \mathbb{R}^{m\times r}$ such that $\WSO^TU_r$ is invertible. Applying Lemma~\ref{l_det} to bound $\| (\WSO^TU_r)^{-1}\|_2$ gives the desired result.
\end{proof}

The importance of this result is that the point selection can also be applied in the weighted inner product case, and the resulting error bound  similar as the {DEIM} bound in \S\ref{S:SRRQR}.


\subsection{On the interpolating property and its generalization} 
{Recall that the original DEIM formulation allows pointwise interpolation $\WSO^T\D f = \WSO^Tf$, i.e., the projection $\D f$ and $f$ match exactly for a set of indices $i_1,\dots,i_r$ determined by the columns of $\WSO$. In the case of $W$-DEIM, the following interpolation properties hold.

\begin{proposition}\label{p_w_interp}
Let $\SO^TW\WU$ be invertible and let $\D$ be as in Definition~\ref{zd:eq:DEF:WDEIM}. Then
$\SO^T W \D f = \SO^TWf$.
\end{proposition}

This can be readily verified; since $\SO^T W\WU$ is invertible then
\[ \SO^T W \D f = (\SO^TW\WU)(\SO^TW\WU)^{-1}\SO^TWf  = \SO^TWf.\]
With the choice $\SO = L^{-T}\WSO$, Proposition~\ref{p_w_interp} simplifies to 
 \begin{equation}\label{eq:W-interpolation}
  \WSO^T (L^T \D f)=\WSO^T (L^T f) .
  \end{equation}
  Hence,  $W$-DEIM cannot in general interpolate $f\in\R^m$ at the selected indices $f_{i_j}=\phi_{i_j}(x_{i_j})$, $j=1,\ldots, r$. An exception to this  is the case that  $W$ has diagonal entries, see \S \ref{SSS::W=diag} for details. But, in many applications the discretized functions values 
may not be available through point evaluation either because there is no analytical expression or they may be sensor data corrupted by noise. In those cases, pointwise interpolation may not be possible, nor desirable -- for a most illuminating discussion see \cite{GEIM}. 
  
\subsubsection{DGEIM}
The DEIM is a realization of the discrete version of the Empirical Interpolation Method (EIM) \cite{EIM} in which interpolation was handled by only using pointwise function evaluation.  
In the same way we can interpret the interpolation condition~\eqref{eq:W-interpolation} as a discrete version of GEIM, DGEIM, as a particular case of $W$-DEIM.


{To this end, consider a more general concept of interpolation using a } family of linear functionals, see \cite[Chapter 11]{Deutsch-BestApprInnPS-book}.
  Introduce in \eqref{eq:W-interpolation} a column partition of $L = \begin{pmatrix} \ell_1 & \dots & \ell_m \end{pmatrix}$ and rewrite it as 
  \begin{equation}\label{zd:eq:gen_interpol}
  \WSO^T \left( \begin{smallmatrix} \ell_1^T\D f \cr \vdots \cr \ell_m^T\D f\end{smallmatrix}\right) = 
  \WSO^T \left( \begin{smallmatrix} \ell_1^T f \cr \vdots \cr \ell_m^T f\end{smallmatrix}\right),\;\;\mbox{i.e.,}\;\;
  \ell_{i_j}^T\D f = \ell_{i_j}^T f,\;\;j=1,\ldots, r.
  \end{equation}
  {If} we interpret $\ell_i\in\R^m$ as the discretized Riesz representation of a given linear functional, then (\ref{eq:W-interpolation}) {interpolates the desired function $f$} at selected functionals. (The point interpolation corresponds to using the point evaluation functional, $(\ell_i)_j=W_{ji}=W_{ij}=(\ell_j)_i=\delta_{ij}$, where $\delta_{ij}$ is the Kronecker delta.)


  
  

\subsection{How to ensure sparse selection}
{The original DEIM approximation was computationally efficient because it only required evaluating a small number of components of the vector $f$. However, in the computation of $\D f$, the factor $\WSO^T L^T f$ may, in the worst case, require many, or possibly all, components of $f$. This might make $W$-DEIM computationally inefficient. It is clear that the selection is sparse when the matrix $L$ is sparse. 

 The analysis is subdivided into three different cases. When the weighting matrix is sparse, or diagonal, the Cholesky factor $L$ is also sparse. When $W$ is sparse, reordering the matrix may lead to sparse factors $L$. On the other hand, if $W$ is dense, we must resort to an inexact sparse factorization. These cases are discussed below.}

\subsubsection{Diagonal weighting matrix $W$}\label{SSS::W=diag}
If $W=\mathrm{diag}(w_i)_{i=1}^{m}$, then $$L={W}^{1/2}=\mathrm{diag}(\sqrt{w_i})_{i=1}^{m},$$ {and the computation of} $\D f=W^{-1/2}U_r (\WSO^T U_r)^{-1}\WSO^T {W}^{1/2}f$ requires only the indices $i_1,\ldots, i_r$ of $f$ selected by $\WSO$. Furthermore, in this case {the interpolation condition} (\ref{eq:W-interpolation}) {simplifies to}
$$
{\sqrt{w_{i_j}}} (\D f)_{i_j} = {\sqrt{w_{i_j}}} f_{i_j},\;\;j=1,\ldots, r,
$$
i.e.{,} $\D$ is an interpolating projection.  

\subsubsection{Sparse weighting matrix $W$}
In some cases, the matrix $W$ that defines a discrete inner product is large and sparse, {and possibly contains} additional block structure, see e.g.{,} \cite[\S 5.4]{ROM-SANDIA-2014}. {Examples of sparse weighting matrices are discussed in the section on numerical experiments (Section~\ref{S=Examples}).}

{When $W$ is sparse}, one can {take advantage of} sparse factorization techniques {to} compute a pivoted factorization $\Pi^T W \Pi = {L}_s {L}_s^T$, where the permutation matrix $\Pi$ is determined to produce a sparse Cholesky factor ${L}_s$.
(In fact, {the permutation matrix} $\Pi$ {has the additional benefit of making} ${L}_s$  well conditioned for inversion by trying to improve diagonal dominance.)
Then we factor $W=LL^T$ with $L=\Pi{L}_s$, and we have 
$$\D f = \WU (\WSO^T U_r)^{-1}\WSO^T L_s^T \Pi^T f. $$
Since $\WSO^T (L_s^T \Pi^T)$ will select only a small portion of the rows of a sparse matrix $L_s^T \Pi^T$, the product $\WSO^T L_s^T \Pi^T f$ is expect to require only relatively small number of the entries of $f$. An efficient implementation of this procedure would deploy the data structure and algorithms from the sparse matrices technology.

We now see an advantage of pure algebraic selection of the interpolation indices, as featured in the Q-DEIM version of the method \cite{drmac-gugercin-DEIM-2016}. 
In Q-DEIM, the index selection is computed by a rank revealing (column) pivoting in the QR factorization of the $r\times m$  matrix $U_r^T$, where $r\ll m$. The role of pivoting is to select an $r\times r$ submatrix of $U_r$ with small inverse. Hence, as argued in \cite{drmac-gugercin-DEIM-2016}, it might be possible to find such a submatrix without having to touch all rows of $U_r$. 

One possible way to improve sparsity is to lock certain columns of $U_r^T$ (whose indices correspond to non-sparse rows of $L_s^T$) and exclude them from the pivot selection. Since $m\gg r$, it is very likely that even with some columns of $U_r^T$ excluded, the selection will perform well. In fact, pivoting in the QR factorization can be modified to prefer indices that correspond to most sparse rows of $L_s^T$. 
\subsubsection{General dense positive definite $W$}
In the most difficult case, the natural inner product is defined with large dimensional dense positive definite $W$ that is also difficult to compute. For instance, as mentioned in \S \ref{SS=UPISPS}, $W$ can be the Gramian obtained by solving a large scale Lyapunov equation, or replaced by an empirical approximation based on the method of snapshots.

If computational complexity requires enforcing sparsity of the selection operator, then we can resort to inexact sparse factorization of the form $\Pi^T W \Pi + \delta W = \widetilde{L_s} \widetilde{L_s}^T$, i.e.{,} we compute $W\approx (\Pi\widetilde{L_s}) (\Pi\widetilde{L_s})^T$. {The resulting approximation has }  the backward error $\Delta W = \Pi \delta W\Pi^T$ {as a result of a thresholding} strategy to produce the sparse factor $\widetilde{L_s}$. {We mention two possibilities here.} The incomplete Cholesky factorization is one candidate, see e.g. \cite{Lin-ICHOL}. The matrix $W$ can also be sparsified by zeroing 
entries $W_{ij}$ if e.g.{,} $|W_{ij}|/\sqrt{W_{ii}W_{jj}}$ is below some threshold. 

Let us identify, for simplicity, $W\equiv \Pi^T W\Pi = LL^T$, so that $W+\delta W = \widetilde{L_s} \widetilde{L_s}^T$.
Set $\widetilde{\D}=\WU (\WSO^T U_r)^{-1}\WSO^T \widetilde{L_s}$.
Then 
\begin{align*}
\| \D -\widetilde{\D}\|_W \leq & \>  \|(\WSO^T U_r)^{-1}\|_2 \|\WSO^T (\Id_m - \widetilde{L_s}^T L^{-T})\|_2 \\
= & \> \|(\WSO^T U_r)^{-1}\|_2 \|L^{-1}(L - \widetilde{L_s})\WSO\|_2 .
\end{align*}
Now, from $\widetilde{\D}f = \D f + (\widetilde{\D} - \D)f$ we have
\begin{align*}
\frac{\|f-\widetilde{\D}f\|_W}{\|f\|_W} \leq & \>  \frac{\|f-{\D}f\|_W}{\|f\|_W} + \| \D -\widetilde{\D}\|_W \\
\leq & \> \frac{\|f-{\D}f\|_W}{\|f\|_W} + \|(\WSO^T U_r)^{-1}\|_2 \|L^{-1}(L - \widetilde{L_s})\WSO\|_2.
\end{align*}
One can also justify using the sparsified weighting matrix in a backward sense, i.e. using $W+\delta W$ 
as the generator of the inner product. This line of reasoning via the incomplete factorization requires further analysis 
which we defer to our future work.
{Of course, in the case of dense $W$, saving the work in evaluating $f$ by the generalized interpolation (\ref{zd:eq:gen_interpol}) is nearly impossible as it may require too many entries to be practical. In that case, one can resort to point-wise interpolation that we discuss next.}
%
%

\subsection{Pointwise-interpolating $W$-DEIM}\label{ss_pointwise}
Note that in the formula for the $W$-DEIM projection in Definition \ref{zd:eq:DEF:WDEIM} there is a certain freedom in choosing $\SO$. The key in our formulation is indeed that we have left it as an adaptable device.  
In the case of the original DEIM with $W=\Id_m$, $\SO\equiv \WSO$ is a submatrix of $\Id_m$, resulting in more efficient computation of the projection \cite{DEIM}. If a generalized interpolation of the type (\ref{eq:W-interpolation}) and (\ref{zd:eq:gen_interpol}) is desired, then  $\SO^T = \WSO^T L^{-1}$ as in (\ref{zd:eq:SL}) in Proposition \ref{zd:PROP:SL} will accomplish the task. 

On the other hand, if we want point-wise interpolation 
\begin{equation}\label{e_pointwise}
{\WSO^T\D f = \WSO^T f \qquad \Longleftrightarrow \qquad } (\D f)_{i_j}=f_{i_j} ,\;\;j=1,\ldots , r
\end{equation}
 also  in the weighted case with a general positive definite $W$, then this can be obtained using the following definition. 
\begin{definition}\label{d_pointwise}
Let the weighted selection operator $\SO$ and the corresponding $W$-DEIM projector $\D$, respectively, be defined as
\begin{equation}\label{zd:eq:SL_pointwise}
\SO^T \equiv \WSO^T W^{-1}\qquad \D \equiv \WU (\WSO^T \WU)^{\dagger}\WSO^T .
\end{equation}
Here $\WU$ is $W$-orthogonal and $\WSO$ has columns from the identity matrix $\Id_m$
\end{definition}


Note that the {relations~\eqref{e_pointwise}, $\D \WP_{\WU} = \WP_{\WU},$ and the error estimate (\ref{e_wdeim_inter}) still apply. However, now, the condition number $\| \D\|_W$ will depend on the specific choice of $\WSO$. We now show how to pick the indices that determine the columns of $\WSO$.

%
%
%

The algorithm proceeds as follows. First, as in Algorithm \ref{zd:ALG:POD}, a thin generalized SVD \cite{van1976generalizing} of the $m\times n_s$ snapshot matrix $Y$ is computed and truncated to obtain low rank approximation  $Y \approx \WU\widehat{\Sigma} \widehat{V}^T$, where $\widehat{V}^T \widehat{V}=\Id_r$ and $\WU\in\R^{m\times r}$ is $W$-orthonormal, i.e., $\WU^TW\WU = \Id_r$. 

Then, the thin QR of $\WU = Q_{\WU}R_{\WU}$ is computed, and strong RRQR is applied to $Q_{\WU}^T$, to obtain the selection operator $\WSO$ (whose columns come from the $m\times m$ identity matrix). Finally, we set $\SO \equiv W^{-1} \WSO$. This procedure  is summarized in Algorithm~\ref{zd:ALG:POD_W}, where the first two steps are implemented as in Algorithm \ref{zd:ALG:POD}. The corresponding error bound is given in Theorem \ref{t_dgeim_rrqr_2}.

\begin{algorithm}[hbt]
\caption{$[\WU,\WSO,Q_{\WU}] =\mbox{$W$-POD-DEIM}(Y,W, \eta)$}
\label{zd:ALG:W-POD-DEIM-1}
\begin{algorithmic}[1]
\REQUIRE  Snapshots $Y\in\R^{m\times n_s}$, $n_s<m$. 
Symmetric positive definite $W\in\R^{m\times m}$. 
{Tuning parameter $\eta$.}	
\STATE Compute the thin generalized SVD of $Y$ as $Y = {U_Y} \Sigma V^T$ with ${U_Y^TWU_Y} = \Id_{n_s}$.
\STATE Determine an appropriate index $r$ and define $\WU={U_Y}(:,1:r)$.
\STATE Compute the thin QR factorization of $\WU = Q_{\WU}R_{\WU}$ .
\STATE Apply strong RRQR{~\cite[Algorithm 4]{GuE96} (with parameter $f=\eta$)} to $Q^T_{\WU}$ to give 
		\[ Q^T_{\WU} \begin{pmatrix} \mat{\Pi}_1 & \mat{\Pi}_2\end{pmatrix}  = \mat{Q} \begin{pmatrix} \mat{R}_{11} & \mat{R}_{22} \end{pmatrix},\;\;\Pi= \begin{pmatrix} \mat{\Pi}_1 & \mat{\Pi}_2\end{pmatrix}.\]
		\STATE $\WSO = \mat{\Pi_1}$.
\ENSURE $W$-orthogonal basis $\WU$ (optional), interpolation selection matrix $\WSO$, and orthogonal basis $Q_{\WU}$ (optional), defining 
$$\D = \WU (\WSO^T \WU)^{-1}\WSO^T  \equiv Q_{\WU} (\WSO^T  Q_{\WU})^{-1}\WSO^T .$$
\end{algorithmic}
\label{zd:ALG:POD_W}
\end{algorithm}

\begin{theorem}\label{t_dgeim_rrqr_2}
Assume that the DEIM projection operator $\D$ is defined as in Algorithm \ref{zd:ALG:POD_W}. Then
\begin{equation}\label{zd:eq:W-dgeim-bound}
\| f - \D f\|_W  \leq \sqrt{1 + \eta^2 r(m-r) } \sqrt{\kappa_2(W)} \|  f - \WP_{\WU} f\|_W . 
\end{equation}
\end{theorem}
\begin{proof}
Note that  $\| \D\|_W = \| L^T \D L^{-T}\|_2  \leq \sqrt{\kappa_2(W)} \| \D\|_2.$ We now bound $\| \D\|_2$. Consider the thin QR of $\WU = Q_{\WU}R_{\WU}$, where $R_{\WU}$ must be nonsingular. Then 
$$
 \D =  Q_{\WU}R_{\WU} (\WSO^TQ_{\WU}R_{\WU})^{-1}\WSO^T 
 =  Q_{\WU} (\WSO^TQ_{\WU})^{-1}\WSO^T.
 $$ 
Since $Q_{\WU}$ and $\WSO$ have orthonormal columns, $\|\D\|_2 = \|(\WSO^TQ_{\WU})^{-1}\|_2$. The rest of the proof is similar to Theorem~\ref{t_dgeim_rrqr}. 
\end{proof}

\subsubsection{Scaling invariant error bound}\label{ss_pointwise_scaling}

Note that, compared to Theorem~\ref{t_dgeim_rrqr}, the error bound (\ref{zd:eq:W-dgeim-bound}) has an additional factor of $\sqrt{\kappa_2(W)}$. For highly ill-conditioned matrices $W$, this considerably inflates the error bound and possibly the actual error as well. It is instructive to see how a simple trick can improve this undesirable situation. 

Let $\Delta=\mathrm{diag}(\sqrt{W_{ii}})_{i=1}^m$ and $W_s= \Delta^{-1} W \Delta^{-1}$; {note that this scaling ensures} $(W_s)_{ii}=1$ for all $i=1,\dots,m$. It is well known (see \cite{slu-69}) that this diagonal equilibration nearly minimizes the spectral condition number over all diagonal scalings,  
\begin{equation}\label{zd:eq:Sluis}
\kappa_2(W_s)\leq m \min_{D{\in\mathcal{D}^m}}\kappa_2(DWD),
\end{equation}
{where $\mathcal{D}^m$ is the space of diagonal $m\times m$ matrices. }
The task is to eliminate the scaling factor $\Delta$ from the bound on $\| \D\|_W$ (by {the use of} a different subset selection) and to replace $\sqrt{\kappa_2(W)}$ with $\sqrt{\kappa_2(W_s)}$ -- {which can be} a substantial improvement {for certain applications of interest}. To that end, we must examine how $W$ influences the structure of $\widehat{U}$, and interweave assembling of $\widehat{U}$ with the construction of the DEIM selection operator.  The selection operator is $\SO^T = \WSO^T W^{-1}$, as in Algorithm \ref{zd:ALG:W-POD-DEIM-1}.

We use the expression for the weighted POD basis $\widehat{U}$ as in Algorithm \ref{zd:ALG:POD}, i.e. $\widehat{U}=L^{-T}U_r$, where  $W=LL^T$ and $U_r^T U_r=\Id_r$.
If we define $L_s = \Delta^{-1}L$, then $W_s = L_s L_s^T$; $L_s$ has rows of unit Euclidean length, and, since $\D = \widehat{U}(\WSO^T\widehat{U})^{-1}\WSO$,
$$
L^T \D L^{-T} = U_r (\WSO^T \Delta^{-1}L_s^{-T}U_r)^{-1}\WSO^T \Delta^{-1}L_s^{-T} .
$$
The key observation is that $\WSO^T\Delta^{-1} = \widehat{\Delta}^{-1}\WSO^T$, where $\widehat{\Delta}$ is a diagonal matrix with the vector
$\WSO^T \Delta$ on its diagonal. This cancels out $\Delta$,
$$
L^T \D L^{-T} = U_r (\widehat{\Delta}^{-1}\WSO^T L_s^{-T}U_r)^{-1}\widehat{\Delta}^{-1}\WSO^T L_s^{-T} = 
U_r (\WSO^T L_s^{-T}U_r)^{-1}\WSO^T L_s^{-T} .
$$
Let now $L_s^{-T}U_r = Q_{\WU} R_s$ be the QR factorization. (Note that $L_s^{-T}U_r = \Delta \widehat{U}$.) Then
$$
\D = \Delta^{-1} Q_{\WU} (\WSO^T Q_{\WU})^{-1}\widehat{\Delta}\WSO^T,\;\;
L^T \D L^{-T} = L_s^T Q_{\WU} (\WSO^T Q_{\WU})^{-1}\WSO^T L_s^{-T} ,
$$
and we conclude that DEIM selection using $Q_{\WU}$ yields the desired bound
$$
\| \D\|_W  \leq \|L_s^T\|_2\|L_s^{-T}\|_2 \| (\WSO^T Q_{\WU})^{-1}\|_2 = 
\sqrt{\kappa_2(W_s)} \| (\WSO^T Q_{\WU})^{-1}\|_2 .
$$
These considerations are summarized in Algorithm \ref{zd:ALG:POD_W-s} and Theorem \ref{zd:TM:W-s-DEIM}. 
\begin{algorithm}[hbt]
	\caption{$[\WU,\WSO,Q_{\WU},\Delta, \widehat{\Delta}] =\mbox{$W$-$\Delta$-POD-DEIM}(Y,W\equiv LL^T,\eta)$}
	\label{zd:ALG:W-POD-DEIM-2}
	\begin{algorithmic}[1]
		\REQUIRE  Snapshots $Y\in\R^{m\times n_s}$, $n_s<m$. Symmetric positive definite $W\in\R^{m\times m}$.
		{Tuning parameter $\eta$.}	
		\STATE Compute the thin SVD of $L^T Y$ as $L^T Y = {U} \Sigma V^T$. \COMMENT{$Y=(L^{-T}U)\Sigma V^T$ is a GSVD of $Y$, with $W$-orthogonal $L^{-T}U$ and orthogonal $V$.}
		\STATE Determine an appropriate index $r$ and define $U_r={U}(:,1:r)$.
		\STATE $\Delta=\mathrm{diag}(\sqrt{W_{ii}})_{i=1}^m$ ; $L_s = \Delta^{-1} L$.
		\STATE Compute the thin QR factorization of $L_s^{-T}U_r$ as $L_s^{-T}U_r = Q_{\WU}R_{s}$ .
		\STATE Apply strong RRQR {~\cite[Algorithm 4]{GuE96} (with parameter $f=\eta$)} to $Q^T_{\WU}$ to give 
		\[ Q^T_{\WU} \begin{pmatrix} \mat{\Pi}_1 & \mat{\Pi}_2\end{pmatrix}  = \mat{Q} \begin{pmatrix} \mat{R}_{11} & \mat{R}_{22} \end{pmatrix},\;\;\Pi= \begin{pmatrix} \mat{\Pi}_1 & \mat{\Pi}_2\end{pmatrix}.\]
		\STATE $\WSO = \mat{\Pi_1}$; $\widehat{\Delta}= \mathrm{diag}(\WSO^T \mathrm{diag}(W))$.
		\ENSURE $W$-orthogonal basis $\WU=L^{-T}U_r$ (optional), interpolation selection matrix $\WSO$, diagonal matrices $\Delta$, $\widehat{\Delta}$ (optional), and orthogonal basis $Q_{\WU}$ (optional), defining 
		$$\D = \WU (\WSO^T \WU)^{-1}\WSO^T  \equiv \Delta^{-1} Q_{\WU} (\WSO^T  Q_{\WU})^{-1} \widehat{\Delta}\WSO^T \equiv \Delta^{-1} Q_{\WU} (\WSO^T  Q_{\WU})^{-1} \WSO^T \Delta.$$
	\end{algorithmic}
	\label{zd:ALG:POD_W-s}
\end{algorithm}

\begin{theorem}\label{zd:TM:W-s-DEIM}
	Assume that the DEIM projection operator $\D$ is defined as in Algorithm \ref{zd:ALG:POD_W-s}. Then
	\begin{equation}\label{zd:eq:W-s-dgeim-bound}
	\| f - \D f\|_W  \leq \sqrt{1 + \eta^2 r(m-r) } \sqrt{\kappa_2(W_s)} \|  f - \WP_U f\|_W . 
	\end{equation}
\end{theorem}

\begin{remark}
	{\em
		It follows from (\ref{zd:eq:Sluis}) that the DEIM projection error bound (\ref{zd:eq:W-s-dgeim-bound}) that applies to Algorithm \ref{zd:ALG:POD_W-s}
		is never much worse ($\sqrt{\kappa_2(W_s)}\leq \sqrt{m}\sqrt{\kappa_2(W)}$) and it is potentially substantially better\footnote{Take e.g. diagonal and highly ill-conditioned $W$.} ($\sqrt{\kappa_2(W_s)} \ll \sqrt{\kappa_2(W)}$) than the estimate (\ref{zd:eq:W-dgeim-bound}) that holds for Algorithm \ref{zd:ALG:POD_W}.
		Although the two algorithms determine $\WSO$ from different orthonormal matrices, the factor $\sqrt{1 + \eta^2 r(m-r)}$ is the same, because of the property of the sRRQR. 
	}
\end{remark}
\begin{remark}
	{\em
In both Algorithm \ref{zd:ALG:POD_W-s} and Algorithm \ref{zd:ALG:POD_W}, the sRRQR and computation of $\WSO$ can be replaced with the Q-DEIM selection \cite{drmac-gugercin-DEIM-2016}, which is more efficient, essentially nearly as robust, but with weaker theoretical bound. However, the weaker upper bound on $\kappa$ is unlikely to make a substantial difference in practical computations, and both algorithms can be implemented using Q-DEIM.
}
\end{remark}
\begin{remark}
	{\em
		For better numerical properties, the Cholesky factorization can be computed with pivoting, $\Pi^T W \Pi = LL^T$, i.e. $W = (\Pi L)(\Pi L)^T$, and we can easily modify Algorithm \ref{zd:ALG:POD_W-s} to work implicitly with $\Pi L$ instead of $L$.		
	}
\end{remark}
\begin{remark}
	{\em
		Note that the computation in Line 1. of Algorithm \ref{zd:ALG:W-POD-DEIM-2} can be rephrased as the GSVD of $Y$, $Y = U_Y \Sigma V^T$, where $U_Y = L^{-T}U$ is $W$-orthogonal, $U_Y^T W U_Y=\Id_m$; see Algorithm \ref{zd:ALG:W-POD-DEIM-1}. Then the matrix $\widehat{U}$ optionally returned by Algorithm 	\ref{zd:ALG:W-POD-DEIM-2} is $\widehat{U}=U_Y(:,1:r)=L^{-T}U_r$. Since $L_s^{-T}=\Delta L^{-T}$, the matrix $L_s^{-T}U_r$ in Line 4. can be expressed as $L_s^{-T}U_r = \Delta L^{-T}U_r=\Delta\widehat{U}$.	
	}
\end{remark}

\section{Numerical Examples}\label{S=Examples}
In this section, we show numerical examples that highlight the benefits of our proposed algorithms. 
\subsection{Example 1}
This example is based on~\cite[Example 3.1]{drmac-gugercin-DEIM-2016}. In this example we study the performance of sRRQR~\cite[Algorithm 4]{GuE96} for subset selection compared to the DEIM approach~\cite{DEIM} and Q--DEIM~\cite{drmac-gugercin-DEIM-2016}. Therefore, we let the weighting matrix $W=\Id_m$. Let 
\begin{equation}\label{e_func_ex1}
 {f}(t;\mu) = 10\exp(-\mu t) \left(\cos(4\mu t) + \sin(4\mu t)\right), \qquad 1 \leq t \leq 6, \;\;\; 0 \leq \mu \leq \pi. 
\end{equation}
The snapshot set is generated by taking $40$ evenly spaced values of $\mu$ and $n=10,000$ evenly spaced points in time. The snapshots are collected in a matrix of size $10000\times 40$, the thin SVD of this matrix is computed and the left singular vectors corresponding to the first $34$ modes are used to define $U_r$. 

To test the interpolation accuracy, we compute its value using the DEIM approximation at $200$ evenly spaced points in the $\mu$-domain. Three different subset selection procedures were used: DEIM, Pivoted QR labeled Q--DEIM, and sRRQR. In each case, we report the relative error defined as 
\[ \text{Rel Err}(\mu_j) \> \equiv \> \frac{\|f_{\mu_j} - \D f_{\mu_j} \|_2}{\| f_{\mu_j}\|_2} \qquad j = 1,\dots,200.\] The results of the comparison are provided in Figure~\ref{f_example1}. 
\begin{figure}[!ht]\centering
$\quad$
\includegraphics[scale=0.3]{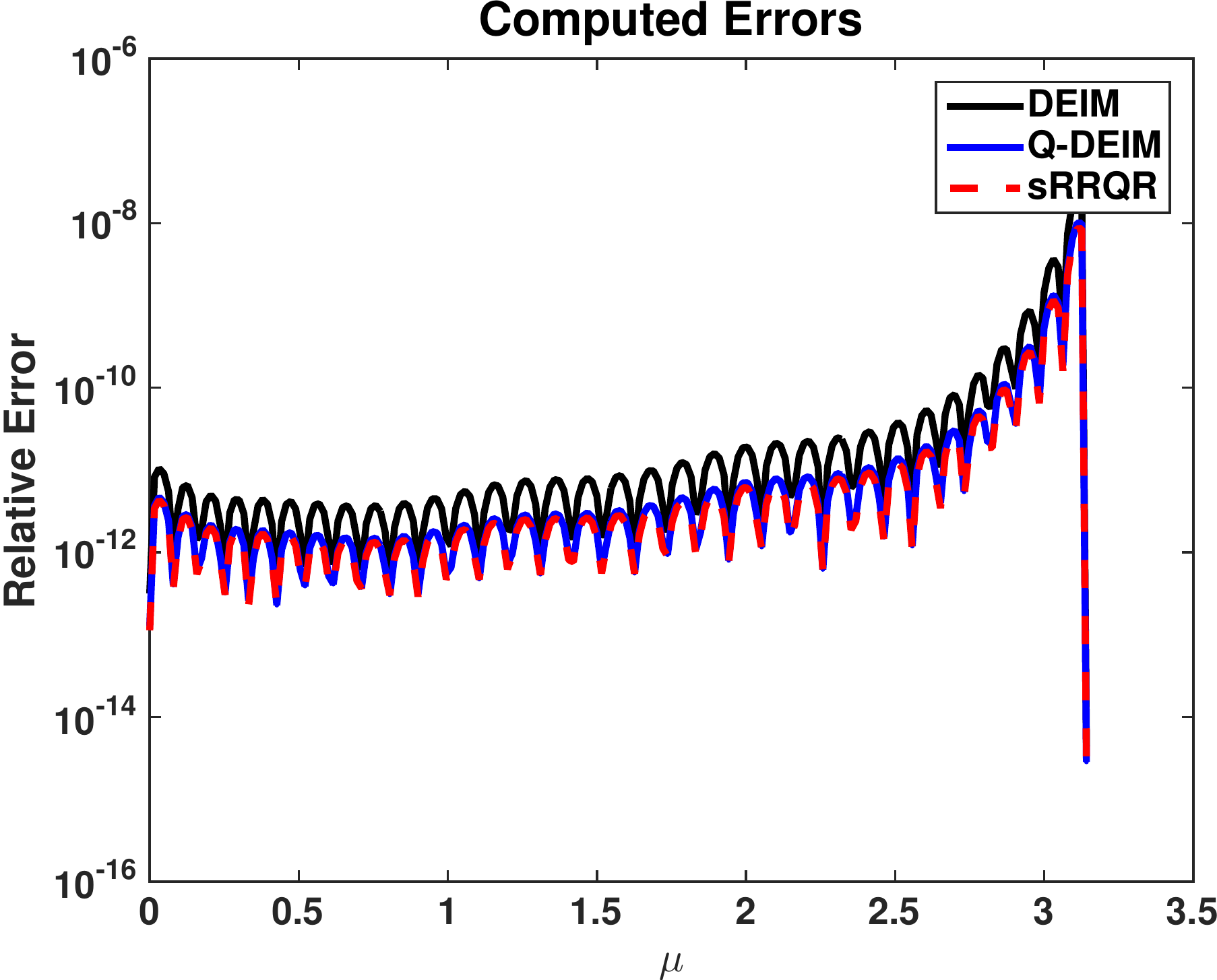}
\includegraphics[scale=0.3]{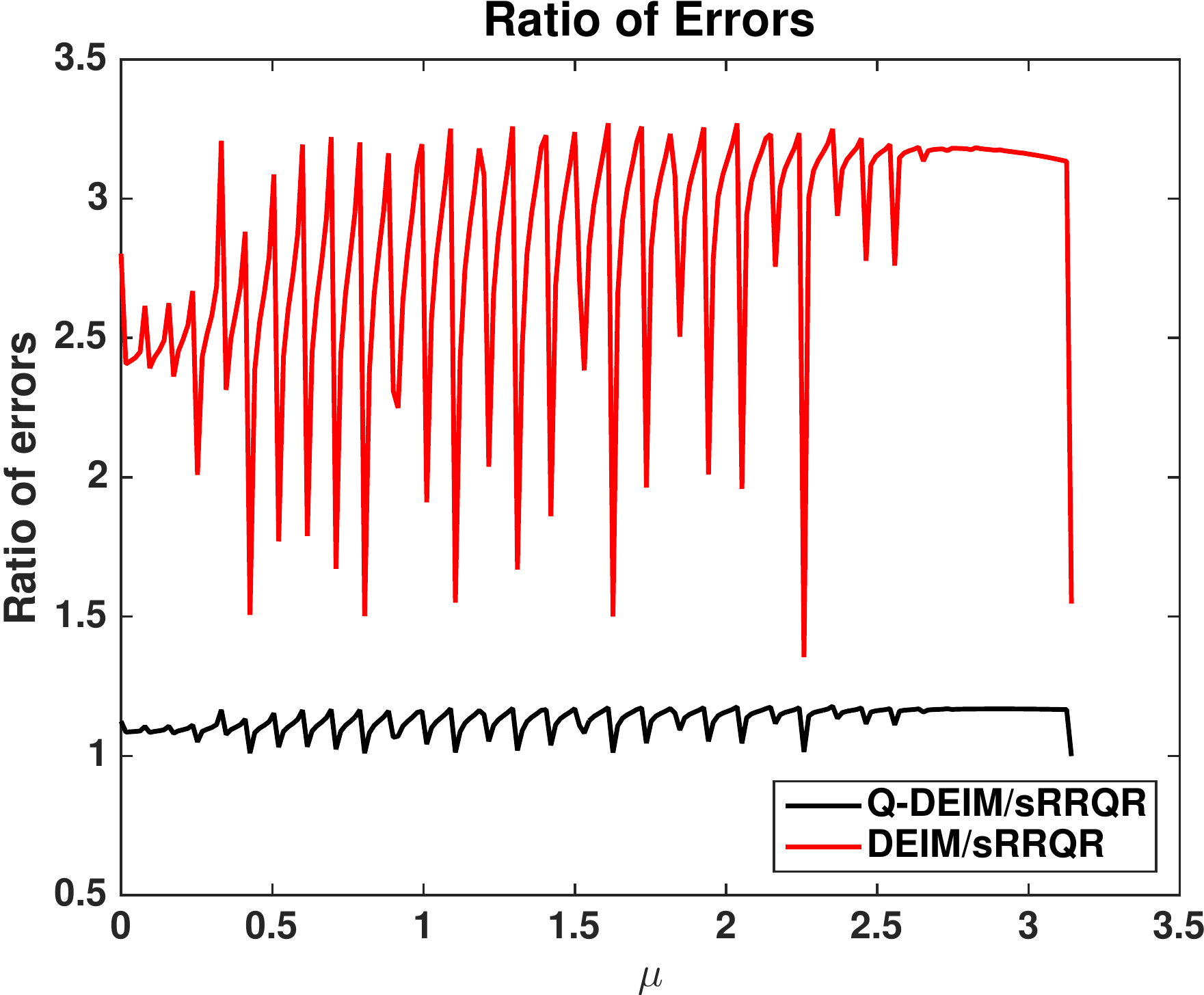}
\caption{Comparison of the approximation errors used to approximate~\eqref{e_func_ex1}. (left) The relative errors are plotted for different subset selection scheme. (right) Ratio of relative errors of (1) Q--DEIM and sRRQR, and (2) DEIM and sRRQR.  }
\label{f_example1}
\end{figure}
We observe that while all three methods are very accurate, Q--DEIM and sRRQR are much more accurate compared to DEIM for this example. Furthermore, from the right plot in Figure~\ref{f_example1}, we see that sRRQR is more accurate compared to both Q--DEIM and sRRQR. In practice, the performance of sRRQR is very similar to Q--DEIM, except for some adversarial cases in which Q--DEIM can fail spectacularly. In the subsequent examples, we use sRRQR for subset selection.

\subsection{Example 2}
Our next example is inspired by the Nonlinear RC-Ladder circuit, which is a standard benchmark problem for model reduction (see, for example~\cite[Section 6]{condon2004empirical}). The underlying model is given by a dynamical system of the form
\[ D\frac{dx(t)}{dt} = \begin{pmatrix}-g(x_1(t)) - g(x_1(t)-x_2(t)) \\ g(x_1(t)-x_2(t)) - g(x_2(t)-x_3(t)) \\ \vdots \\ g(x_{N-1}(t)-x_N(t)) \end{pmatrix} + \begin{pmatrix} u(t) \\ 0 \\ \vdots \\ 0 \end{pmatrix}, \]
where $g(x) = \exp(40 x) + x - 1$ and $u(t) = \exp(-t)$ and $N=1000$. The diagonal matrix $D$ is chosen to have entries 
\[ D_{ii} = \left\{ \begin{array}{ll} 1 & 251\leq i \leq 750 \\ \frac{1}{2} & \text{otherwise}\end{array}\right.\]
The diagonal matrix $D$ induces the norm $\| \cdot\|_D$ and the relative error between the full and the reduced order models are measured in this norm.
\begin{figure}[!ht]\centering
$\quad$
\includegraphics[scale=0.3]{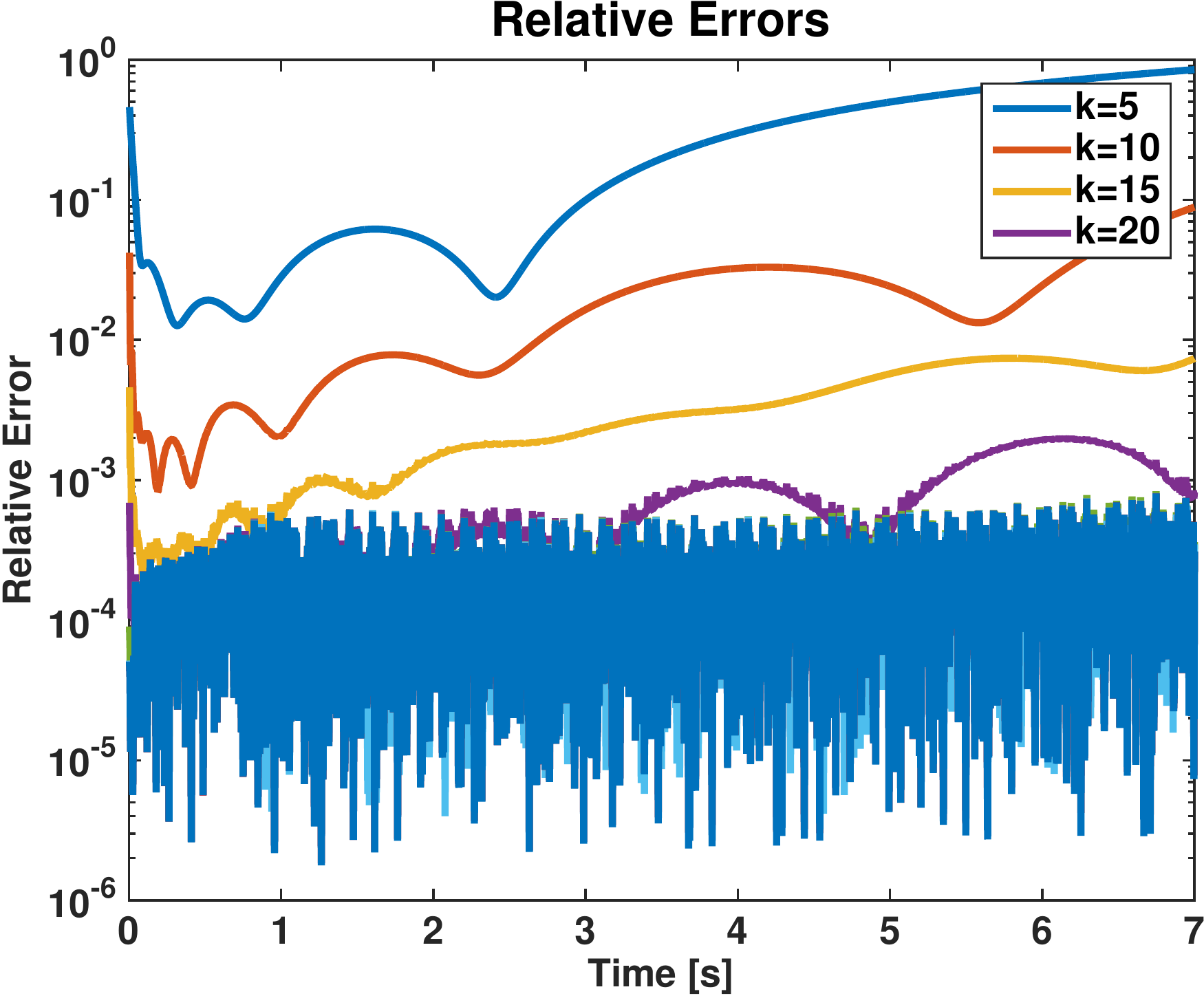}
\includegraphics[scale=0.3]{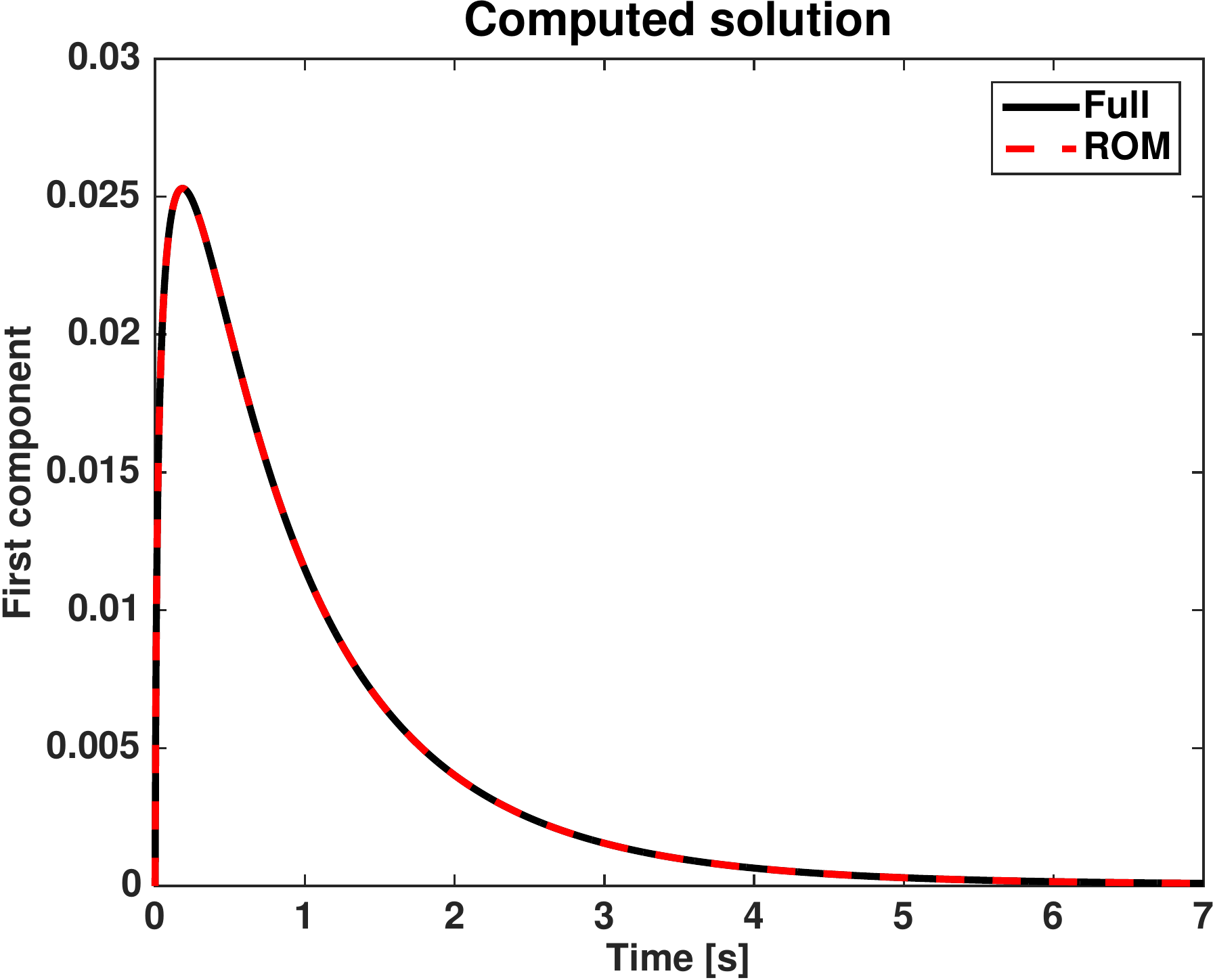}
\caption{The plots refer to Example 2. (left) Relative error of the full and reduced order systems for different times, as a function of number of basis vectors. (right) The $W$-DEIM based reconstruction of the first component $x_1(t)$ as a function of time with $k=40$. }
\label{f_example2}
\end{figure}

The dynamical system is simulated over $t=[0,7]$ seconds and $2000$ snapshots of the dynamical system and the nonlinear function are collected with equidistant time steps. Based on the decay rate of the snapshots, we vary the number of basis vectors from $5$ to $40$. The relative error is defined to be 
\[  \text{Rel. Err.}(t) \equiv \frac{\|x(t)-\hat{x}(t) \|_D}{\|x(t)\|_D},\]
where $x(t)$ is the solution of the dynamical system at time $t$, whereas $\hat{x}(t)$ is the reduced order approximation at the same time. The relative error as a function of number of retained basis vectors is plotted in left panel of Figure~\ref{f_example2}. On the right, the reconstruction of the first component of the dynamical system $x_1(t)$ is shown; here $k=40$ basis vectors were retained. As can be seen, the reconstruction error is low and the $W$-DEIM, indeed, approximates the large-scale dynamical system accurately.   

\subsection{Example 3}
This example is inspired by~\cite[Section 2.3]{peherstorfer2014localized}. The spatial domain is taken to be $\Omega = [0,1]^2$ and the parameter domain is $\mathcal{D} = [0,1]^2$. We define a function $g : \Omega \times \mathcal{D} \rightarrow \mathbb{R}$ which satisfies
\[ g(x_1,x_2;\mu_1,\mu_2) \equiv \frac{1}{\sqrt{h(x_1;\mu_1) + h(x_2;\mu_2) + 0.1^2 }}.\]  
where $h(z;\mu) = ((1-z)-(0.99\cdot\mu-1))^2 $. The function that is to be interpolated is 
\begin{eqnarray}
 f({x};{\mu}) = &  g(x_1,x_2;\mu_1,\mu_2) + g(1-x_1,1-x_2; 1-\mu_1, 1-\mu_2) \\
& + g(1-x_1,x_2; 1-\mu_1,\mu_2) + g(x_1,1-x_2; \mu_1, 1-\mu_2). 
\end{eqnarray}
Depending on the parameter $\mu$, it has a sharp peak in one of the four corners of $\Omega$. The function is discretized on a $100\times 100$ grid in $\Omega$, and parameter samples are drawn from a $25\times 25$ equispaced grid in $\mathcal{D}$. These $625$ snapshots are used to construct the DEIM approximation. We choose three different weighting functions: $W_1$ is the identity matrix, $W_2$ is the weighting matrix corresponding to the $L^2(\Omega)$ inner product, and $W_3$ is the weighting matrix corresponding to the $H^1(\Omega)$ inner product. 
\begin{figure}[!ht]\centering
$\quad$
\includegraphics[scale=0.3]{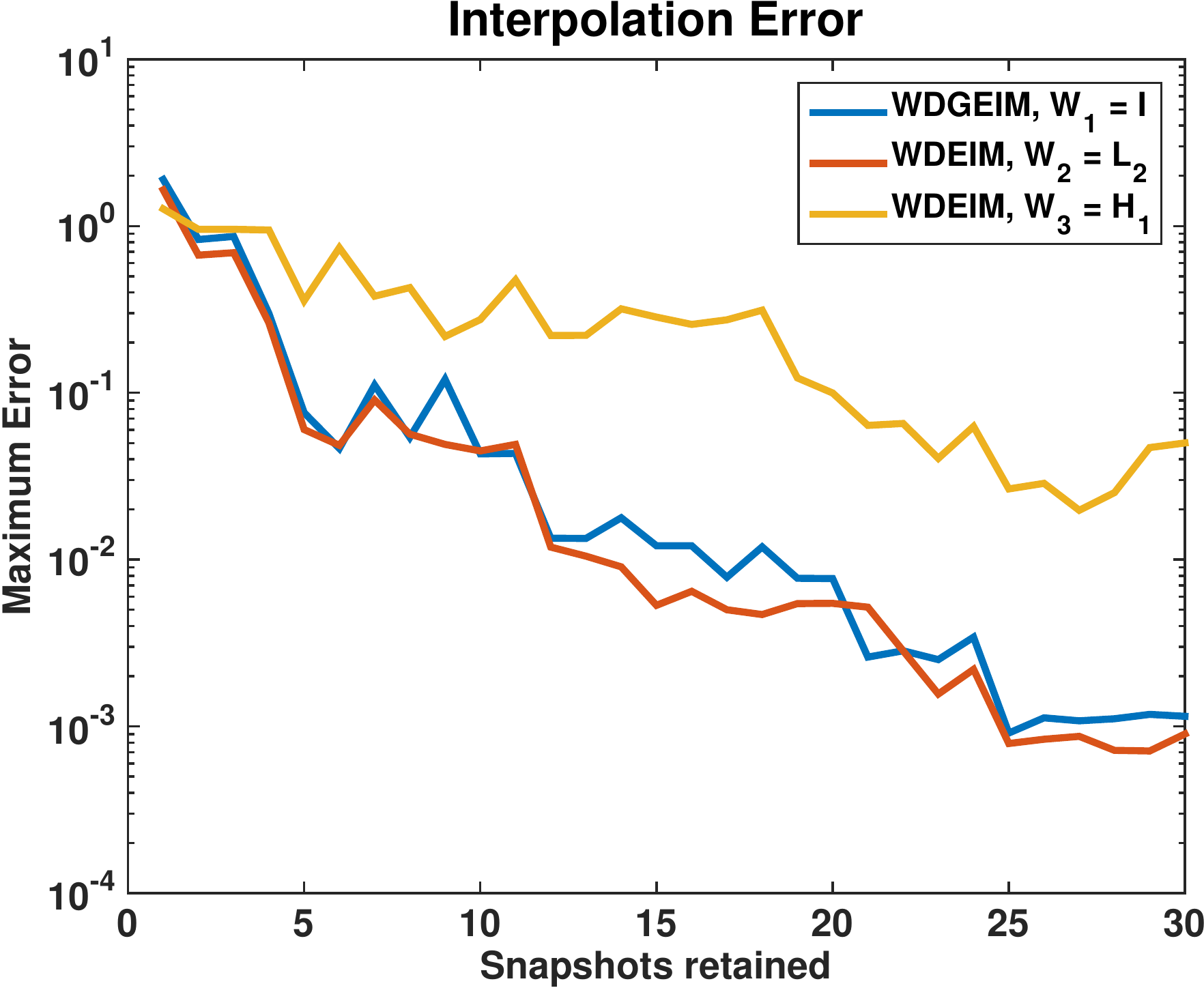}
\includegraphics[scale=0.3]{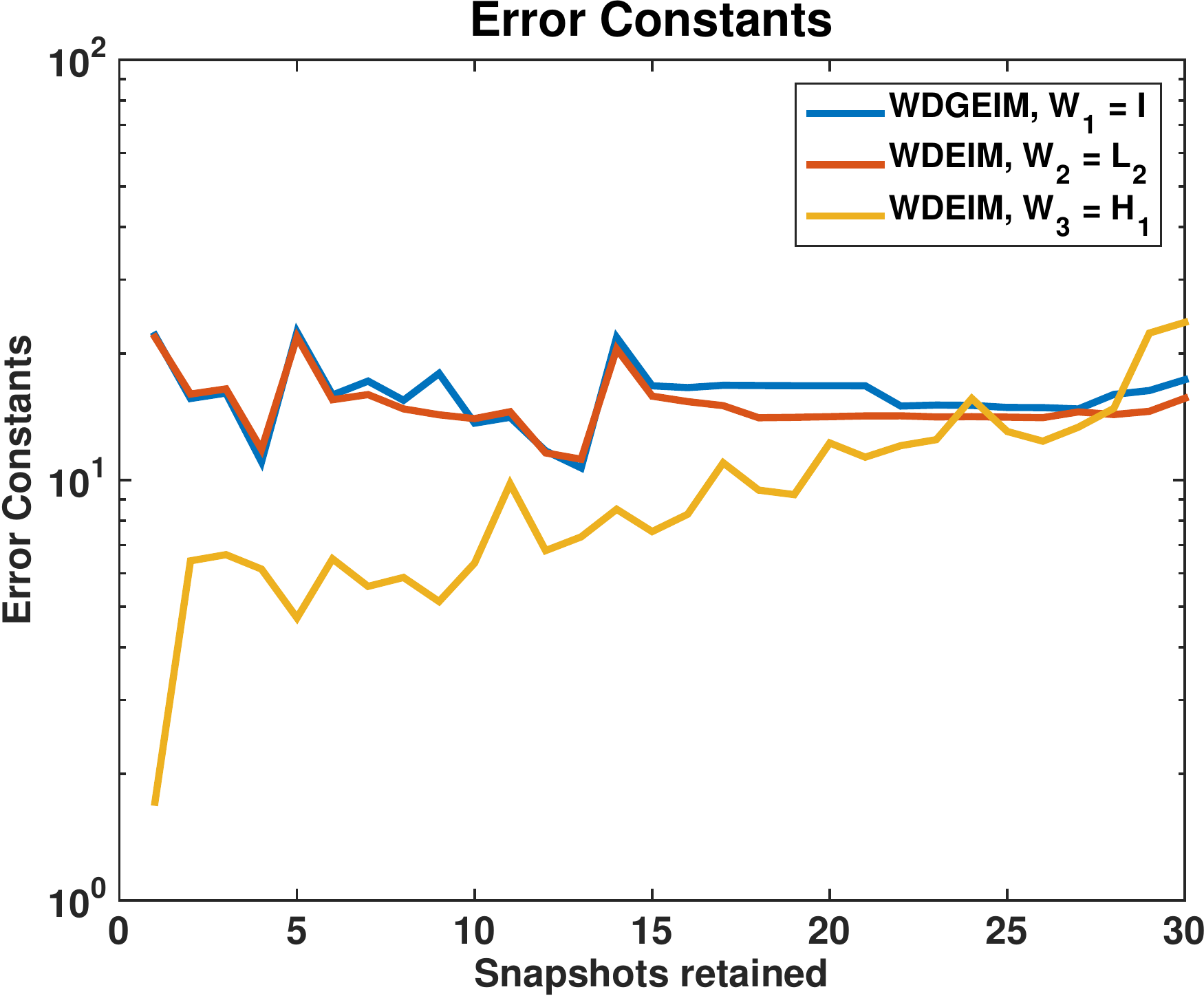}
\caption{(left) {Maximum} relative error {over the test parameters} as a function of number of basis vectors used in the DEIM approximation. (right) Error constants $\| \D\|_W = \| (\WSO^T U_r)^{-1}\|$. Three different weighting matrices were used:   $W_1$ is the identity matrix, $W_2$ is the weighting matrix corresponding to the $L^2(\Omega)$ inner product, and $W_3$ is the weighting matrix corresponding to the $H^1(\Omega)$ inner product.}
\label{f_example3}
\end{figure}

We then compute the average relative error over a test sample corresponding to a $11\times 11$ equispaced grid in $\mathcal{D}$. The relative error is defined to be 
\[ \text{Rel. Err.}_j = \frac{\| f - \D f\|_{W_j}}{\| f\|_{W_j}} \qquad j=1,2,3.  \]
The POD basis is computed using Algorithm~\ref{zd:ALG:POD}, whereas the subset selection is done using sRRQR~\cite[Algorithm 4]{GuE96}. The results of the interpolation errors as a function of number of DEIM interpolation points retained, is displayed in the left panel of Figure~\ref{f_example3}. On the right hand panel of the same figure, we display the error constants $\| \D\|_W = \| (\WSO^T U_r)^{-1}\|$. 
As can be seen, although the error constants increase with increasing number of basis vectors, the overall interpolation error decreases resulting an effective approximation.

\subsection{Example 4}\label{ss_ex_4} This is a continuation of Example 3. We use the same setup as before; however, we compare the different algorithms for $W$-DEIM. In `Method 1' we use  Algorithm~\ref{zd:ALG:POD} to generate the POD basis, while the subset selection is done using sRRQR~\cite[Algorithm 4]{GuE96}. The error constant for this method is $\eta_1 \equiv \| (\WSO^T U_r)^{-1}\|_2$. In `Method 2' we use Algorithm~\ref{zd:ALG:POD_W} with error constant $\eta_2 \equiv \sqrt{\kappa_2(W)} \| (\WSO^TQ_{\WU})^{-1}\|_2$ and in `Method 3' we use Algorithm~\ref{zd:ALG:POD_W-s} with error constant $\eta_3 \equiv \sqrt{\kappa_2(W_s)} \| (\WSO^TQ_{\WU})^{-1}\|_2$.

In Algorithm~\ref{zd:ALG:POD_W}, the GSVD of the snapshot matrix w.r.t. the weighting matrix $W$  was computed as follows. First, the weighted QR was computed using~\cite[Algorithm 2]{lowery2014stability} to obtain $Y = Q_Y R_Y$. Note that $Q_Y^TWQ_Y = \Id_{n_s}$. Then the SVD of $R_Y$ is computed as $R_Y = U_R\Sigma V^T$. We obtain the GSVD of $Y = {U}_Y \Sigma V^T$, where now ${U}_Y = Q_Y U_R$.  

\begin{figure}[!ht]\centering
$\quad$
\includegraphics[scale=0.3]{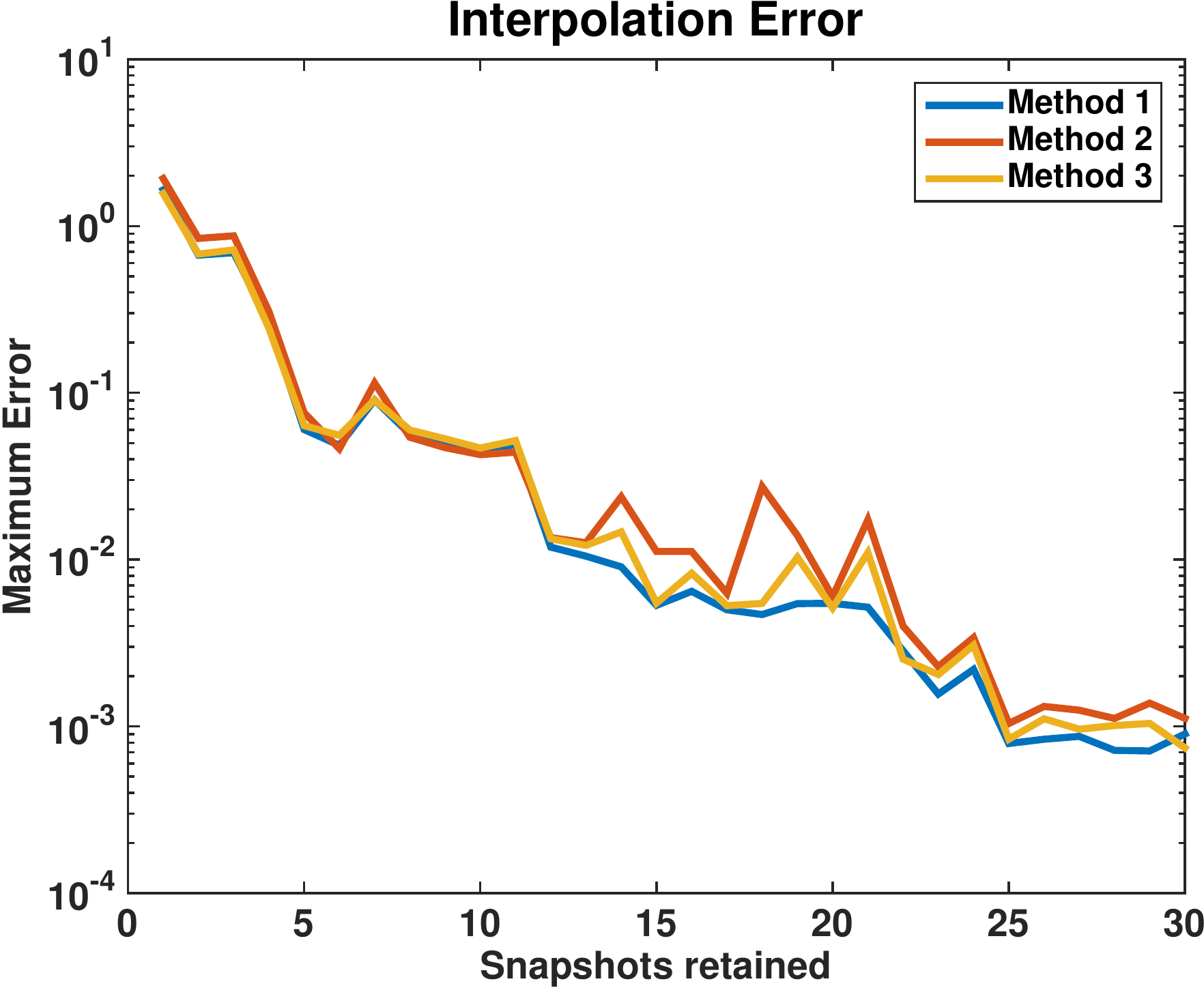}
\includegraphics[scale=0.3]{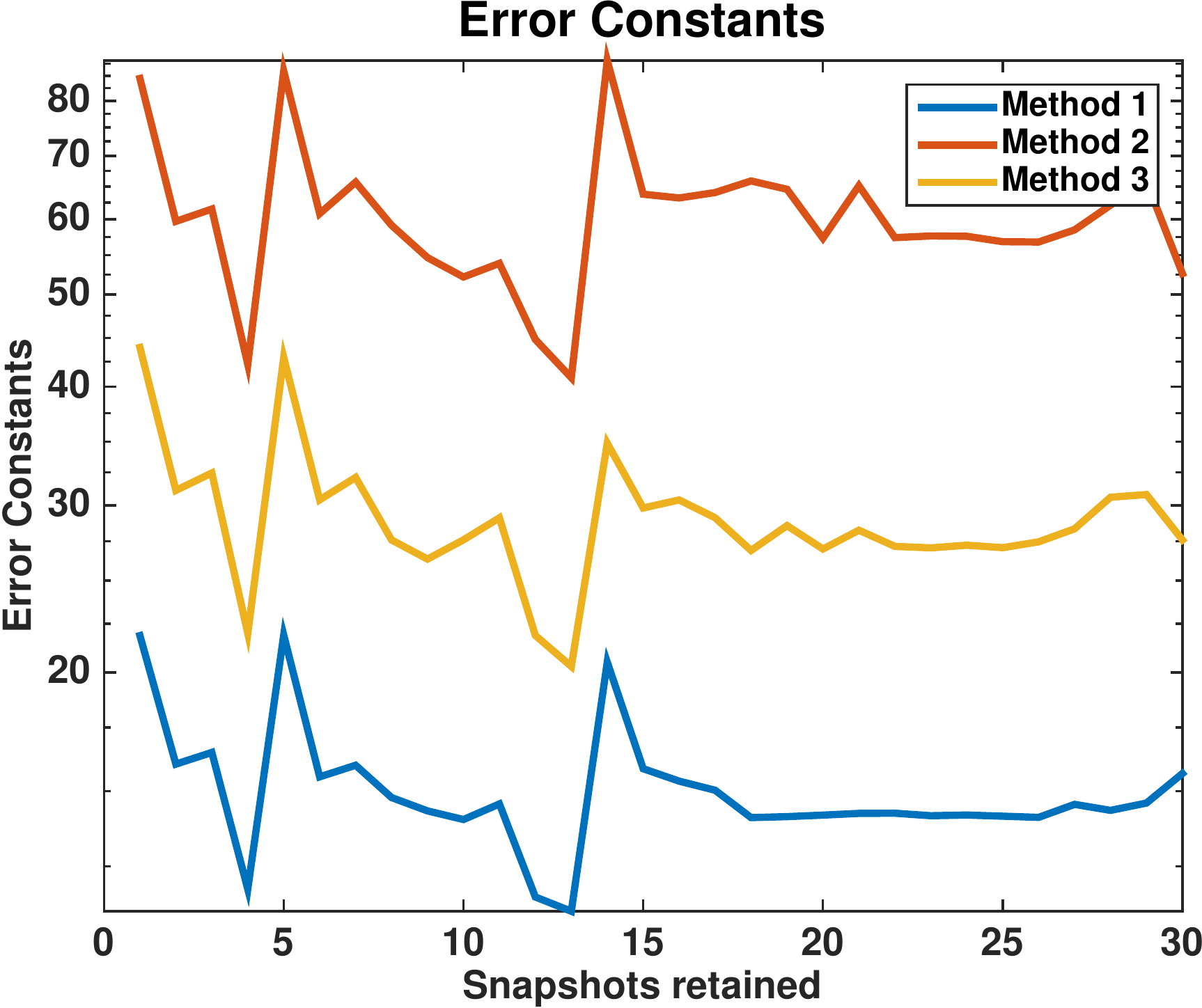}
\caption{(left) Maximum relative error as a function of number of basis vectors used in the DEIM approximation. (right) Error constants for the three methods as defined in Section~\ref{ss_ex_4}. }
\label{f_example4a}
\end{figure}

\begin{figure}[!ht]\centering
$\quad$
\includegraphics[scale=0.3]{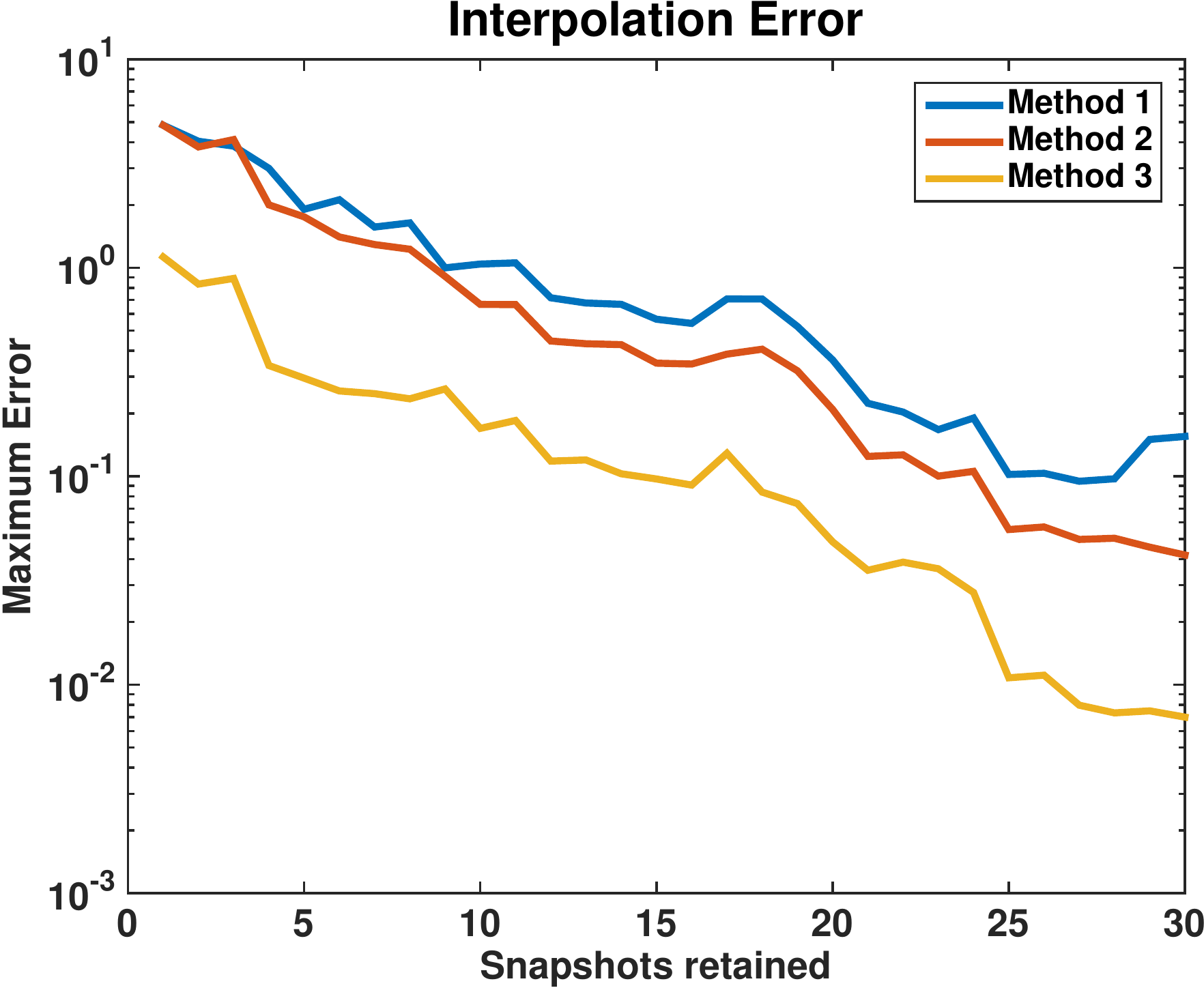}
\includegraphics[scale=0.3]{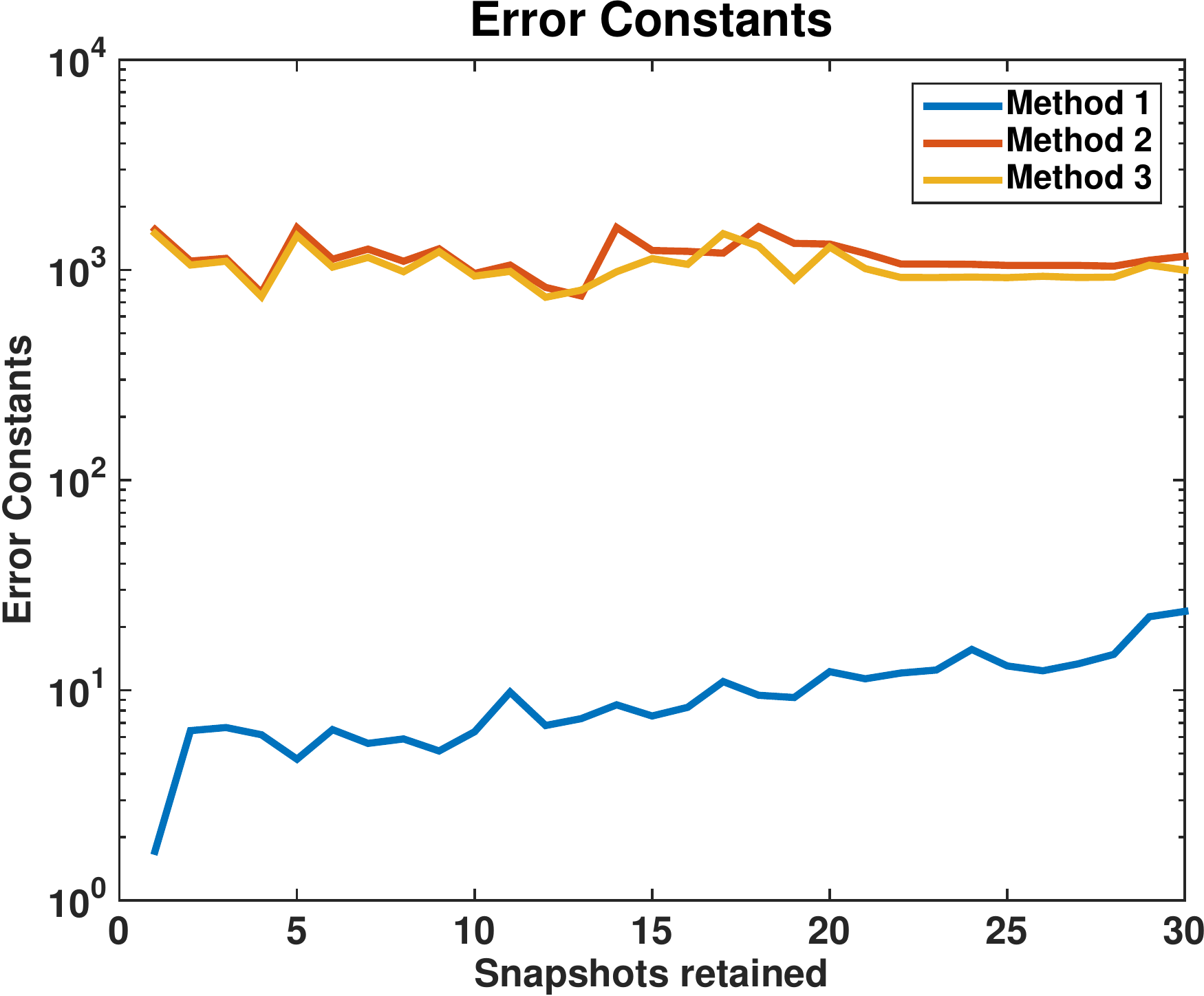}
\caption{(left) Maximum relative error as a function of number of basis vectors used in the DEIM approximation. (right) Error constants for the three methods as defined in Section~\ref{ss_ex_4}. }
\label{f_example4b}
\end{figure}

For a given weighting matrix, we define the relative error as 
\[ \text{Rel. Err.}_j = \frac{\| f - \D_j f\|_{W}}{\| f\|_{W}} \qquad j=1,2,3. \]
The DEIM operators $\D_j$ correspond to the different Methods described above. In Figure~\ref{f_example4a} we plot the relative error using the DEIM approximation and error constants; here the weighting matrix $W= W_2$ corresponds to the $L^2(\Omega)$ inner product. As can be seen, the overall interpolation error from all three methods are comparable. However, the error constants for Method 2 are highest as expected, since it involves $\sqrt{\kappa_2(W)}$. In Figure~\ref{f_example4b} we repeat the same experiment; however, the weighting matrix $W= W_3$ corresponds to the $H^1(\Omega)$ inner product. Our conclusions are similar to the previous weighting matrix. Note here that $W_3$ is more ill-conditioned than $W_2$ and furthermore, for $W= W_3$ we have that $\kappa_2(W) \approx \kappa_2 (W_s)$. Therefore, the difference between the error constants for Methods 2 and 3 is very small. 

In conclusion, for the application at hand, all three $W$-DEIM methods produce comparable results. Methods 2 and 3 maybe desirable if factorization of $W$ is computationally expensive, or even infeasible.

\subsection{Example 5} In this example, we consider a parameterized PDE based on~\cite[Section 8.5]{quarteroni2015reduced}. Consider the following parameterized PDE form defined on domain $\Omega = [0,1]^2$ with boundary $\partial \Omega$
\begin{eqnarray}
- \Delta u  + \boldsymbol{b}(\mu_1) \cdot \nabla u = &  s(\mathbf{x};\boldsymbol{\mu})  & \mathbf{x} \in \Omega\\
 \mathbf{n}\cdot \nabla u = & 0 & \mathbf{x} \in \partial \Omega.
\end{eqnarray}
Here $\boldsymbol{\mu} = [\mu_1,\mu_2,\mu_3]$,  and $\mathbf{n}$ is the normal vector. The wind velocity $\mathbf{b}(\mu_1)$ is taken to be as  $\boldsymbol{b} = [\cos\mu_1,\sin\mu_1]$, which is a constant in space but depends nonlinearly on the parameter $\mu_1$. The source term $s(\boldsymbol{\mu})$ has the form of a Gaussian function centered at $(\mu_2,\mu_3)$ and spread $0.25$

\[ s(\mathbf{x};\boldsymbol{\mu}) = \exp\left( -\frac{(x_1-\mu_2)^2 + (x_2-\mu_3)^2}{0.25^2}\right).\]
The goal of this problem is to construct a reduced order model for the solution $u(\mathbf{x};\boldsymbol{\mu})$ in the domain $\Omega$ over the range of parameters $\mu_1 \in [0,2\pi]$, $\mu_2 \in [0.2,0.8]$ and $\mu_3 \in [0.15,0.35]$. A POD based approach is used to reduce the model of the parameterized PDE with DEIM/WDEIM approximation for the source term.

\begin{figure}[!ht]\centering
\includegraphics[scale=0.3]{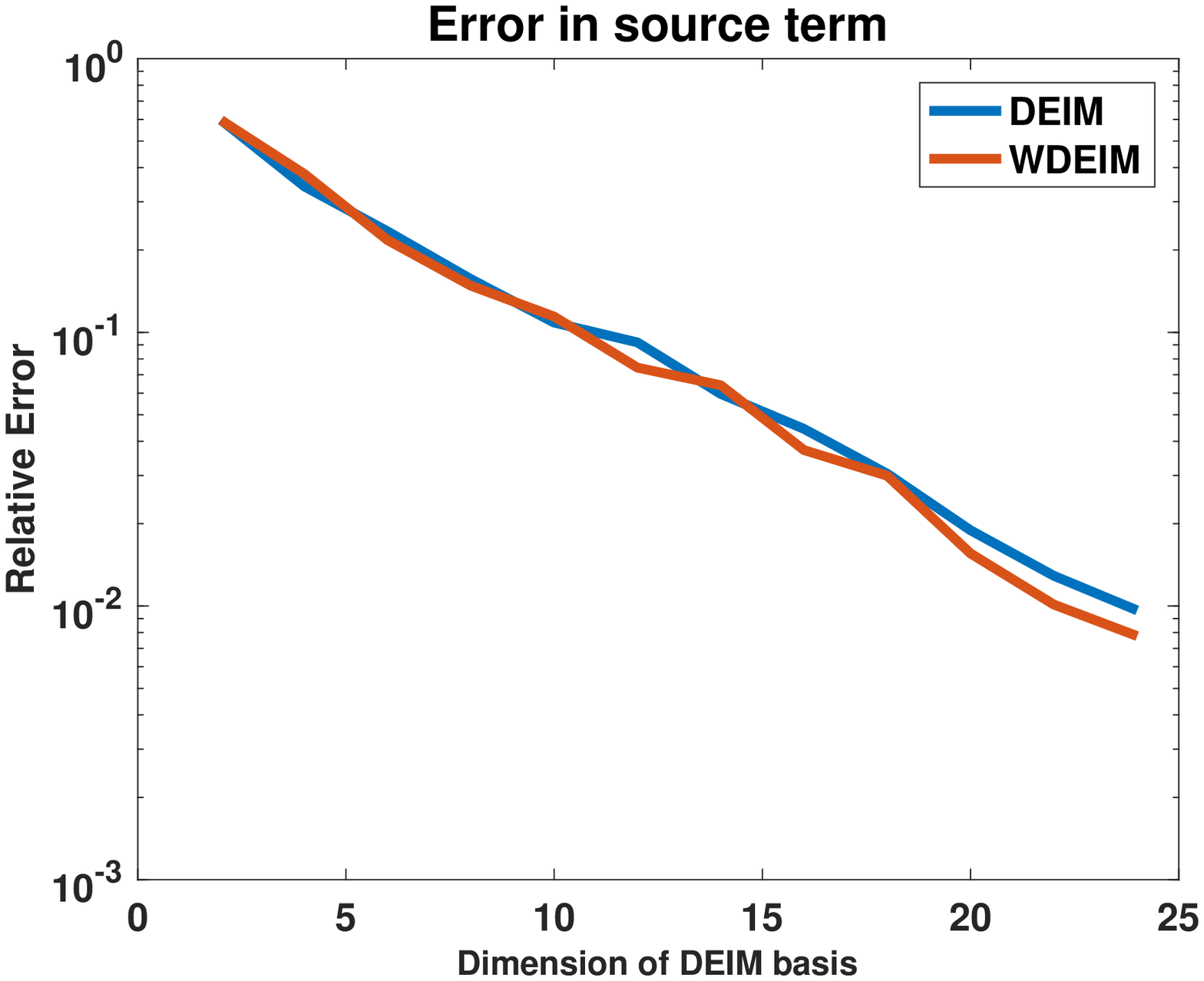}
\includegraphics[scale=0.3]{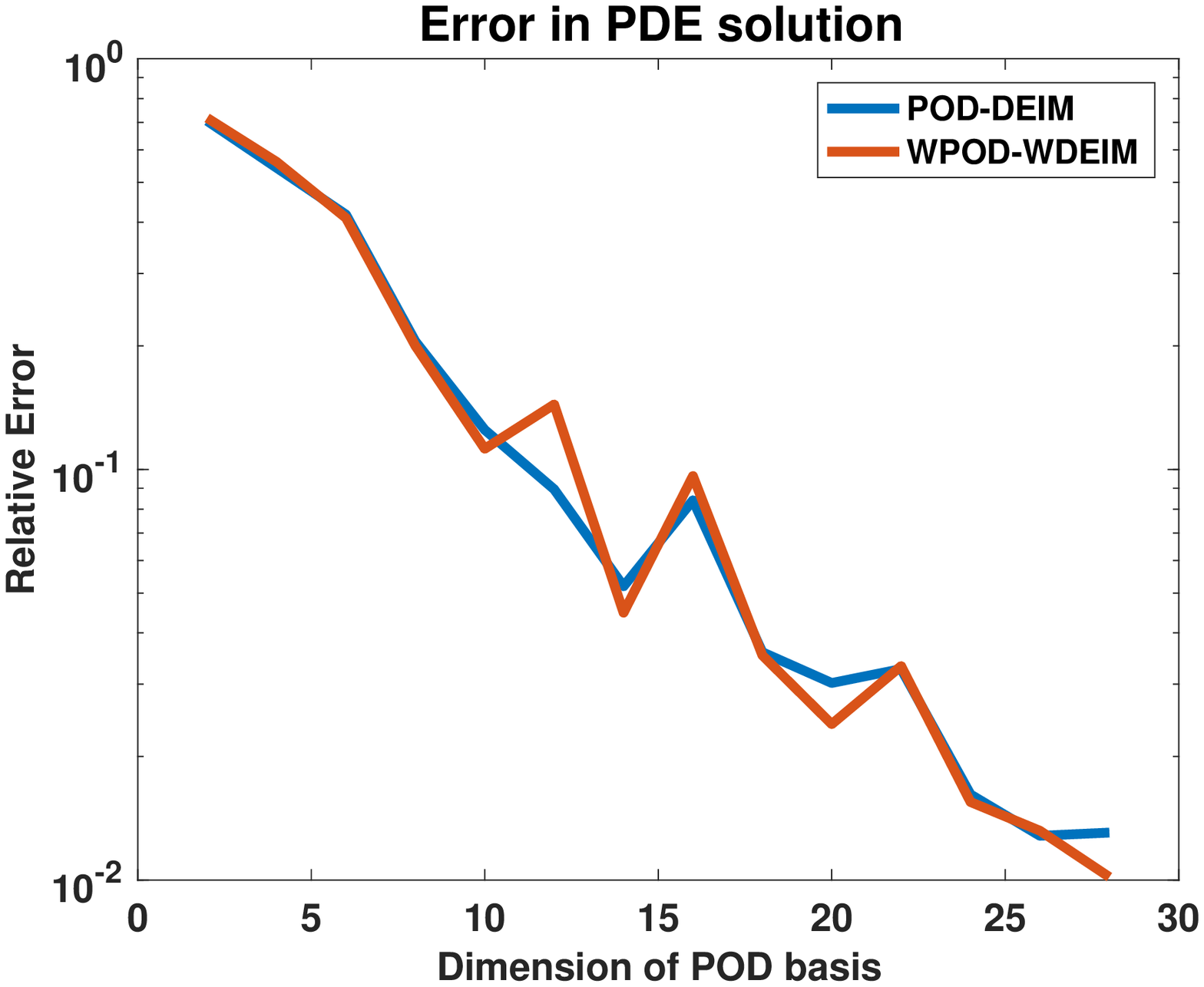}
\caption{(left) Error in the DEIM approximation and WDEIM approximations of the source term $s(\mathbf{x};\boldsymbol{\mu})$. (right)the error in the solution of $u(\mathbf{x};\boldsymbol{\mu})$. In both cases, the error is averaged over $10$ test samples.}
\label{f_ex5}
\end{figure}

As the weighting matrix $W$, we choose the arising from the discrete representation of the $H_1(\Omega)$ inner product. For constructing the WPOD and WDEIM bases, we first generated a training set of parameters $\boldsymbol\mu$ of $1000$ points generated by Latin Hypercube sampling; then the source term and the solution of the PDE is computed at each training point $\boldsymbol\mu$. The maximum dimension for the WPOD and WDEIM bases were chosen to be $20$ and $24$ respectively based on the decay of the singular values. From the same snapshot set we also compute bases for the POD and DEIM with dimensions $20$ and $24$ respectively. For both approaches, we use the PQR for computing for point selections. We report the errors used by both approaches in Figure~\ref{f_ex5}. The errors were averaged over $10$ different randomly generate samples in the parameter range. In the left panel, we compare the error in the DEIM approximation and WDEIM approximations of the source term $s(\mathbf{x};\boldsymbol{\mu})$, whereas in the right panel, we consider the error in the solution of $u(\mathbf{x};\boldsymbol{\mu})$ over the same test samples. For the right panel, the dimension of the DEIM/W-DEIM was chosen to be $24$ and the dimension of the POD/WPOD basis dimension was chosen to be $28$. All the errors were computed with the weighted norm $\|\cdot\|_W$. We see that the error in our approach (WPOD-WDEIM) is comparable with that of the POD-DEIM approach, whereas the error in the WDEIM approach is slightly better than the error in the error in the DEIM approach. 

\section{Conclusions}
The main contributions of this work are: \emph{(i)} it defines a new index selection operator, based on strong rank revealing QR factorization, that nearly attains the optimal DEIM projection error bound; \emph{(ii)} it facilitates the understanding of the canonical structure of the DEIM projection; \emph{(iii)} it establishes a core numerical linear algebra framework for the DEIM projection in weighted inner product spaces; \emph{(iv)}
it defines a discrete version of the Generalized Empirical Interpolation Method (GEIM). We believe that these will be useful for further  development of the DEIM idea and its applications in scientific computing. 

\section*{Acknowledgements} We are indebted to Ilse Ipsen and in particular to the two anonymous referees for constructive criticism and many suggestions that have improved the presentation of the paper.
\bibliography{bibDEIM}
\bibliographystyle{siam}
\end{document}